\documentclass[12pt]{amsart}

\usepackage{amsmath,amssymb,color}
\usepackage{amsthm}
\usepackage{hyperref}
\usepackage{xcolor}
\hypersetup{
	colorlinks,
	linkcolor={blue!90!black},
	citecolor={green!40!black},
	urlcolor={blue!90!black}
}

\usepackage[margin=0.8in]{geometry}

\numberwithin{equation}{section}

\newtheorem{prop}{Proposition}
\newtheorem{lemma}[prop]{Lemma}

\newtheorem{thm}[prop]{Theorem}
\newtheorem{cor}[prop]{Corollary}
\newtheorem{conj}[prop]{Conjecture}
\numberwithin{prop}{section}

\theoremstyle{definition}
\newtheorem{defn}[prop]{Definition}
\newtheorem{ex}[prop]{Example}
\newtheorem{rmk}[prop]{Remark}

\newtheorem{innercustomgeneric}{\customgenericname}
\providecommand{\customgenericname}{}
\newcommand{\newcustomtheorem}[2]{%
	\newenvironment{#1}[1]
	{%
		\renewcommand\customgenericname{#2}%
		\renewcommand\theinnercustomgeneric{##1}%
		\innercustomgeneric
	}
	{\endinnercustomgeneric}
}
\newcustomtheorem{customthm}{Theorem}
\newcustomtheorem{customprop}{Proposition}
\newcustomtheorem{customcor}{Corollary}
\newcommand{\bgb}{\bar\beta}

\newcommand{\mc}{\mathcal}
\newcommand{\R}{\mathbb R}
\newcommand{\C}{\mathbb C}
\renewcommand{\Re}{\mathop{\mathrm{Re}}}

\newcommand{\la}[1]{\left< #1 \right>}

\newcommand{\del}{\partial}
\newcommand{\delb}{\bar{\partial}}\newcommand{\dt}{\frac{\partial}{\partial t}}
\newcommand{\brs}[1]{\left| #1 \right|}

\newcommand{\gD}{\Delta}

\newcommand{\gs}{\sigma}

\newcommand{\gl}{\lambda}

\newcommand{\gw}{\omega}
\newcommand{\ga}{\alpha}
\newcommand{\gb}{\beta}

\renewcommand{\ge}{\epsilon}
\newcommand{\N}{\nabla}

\newcommand{\GG}{\mathbb G}

\newcommand{\til}[1]{\widetilde{#1}}

\renewcommand{\bar}[1]{\overline{#1}}

\renewcommand{\i}{\sqrt{-1}}

\newcommand{\bga}{\bar{\ga}}

\newcommand{\hr}{regularity scale}

\newcommand{\bj}{\bar{j}}

\newcommand{\bl}{\bar{l}}

\newcommand{\bq}{\bar{q}}

\newcommand{\IP}[1]{\left<#1\right>}

\newcommand{\JJ}{\mathbb J}
\newcommand{\J}{\mathbb J}

\newcommand{\HH}{H^2}

\DeclareMathOperator{\Rc}{Rc}
\DeclareMathOperator{\Rm}{Rm}

\DeclareMathOperator{\tr}{tr}
\DeclareMathOperator{\Ker}{Ker}

\DeclareMathOperator{\Id}{Id}

\DeclareMathOperator{\im}{Im}

\DeclareMathOperator{\Ham}{Ham}

\DeclareMathOperator{\End}{End}

\begin{document}

\title[The generalized K\"ahler Calabi-Yau problem]{The generalized K\"ahler Calabi-Yau problem}

\begin{abstract} We formulate an extension of the Calabi conjecture to the setting of generalized K\"ahler geometry.  We show a transgression formula for the Bismut Ricci curvature in this setting, which requires a new local Goto/Kodaira-Spencer deformation result, and use it to show that solutions of the generalized Calabi-Yau equation on compact manifolds are classically K\"ahler, Calabi-Yau, and furthermore unique in their generalized K\"ahler class.  We show that the generalized K\"ahler-Ricci flow is naturally adapted to this conjecture, and exhibit a number of a priori estimates and monotonicity formulas which suggest global existence and convergence.  For initial data in the generalized K\"ahler class of a K\"ahler Calabi-Yau structure we prove the flow exists globally and converges to this unique fixed point.  This has applications to understanding the space of generalized K\"ahler structures, and as a special case yields the topological structure of natural classes of Hamiltonian symplectomorphisms on hyperK\"ahler manifolds.  In the case of commuting-type generalized K\"ahler structures we establish global existence and convergence with arbitrary initial data to a K\"ahler Calabi-Yau metric, which yields a new $\del\delb$-lemma for these structures.
\end{abstract}

\author{Vestislav Apostolov}
\address{Nantes Universit\'e\\BP 92208, 44322, Nantes, France, and Institute of Mathematics and Informatics, Bulgarian Academy of Sciences, Acad. Georgi Bonchev Str., Block 8, 1113 Sofia, Bulgaria}
\email{\href{mailto:Vestislav.Apostolov@univ-nantes.fr}{Vestislav.Apostolov@univ-nantes.fr}}
\author{Xin Fu}
\address{School of Science, Institute for Theoretic Sciences\\ Westlake University\\ Hangzhou 310030, China}
\email{\href{mailto:fux6@uci.edu}{fux6@uci.edu}}
\author{Jeffrey Streets}
\address{Rowland Hall\\
         University of California, Irvine\\
         Irvine, CA 92617}
\email{\href{mailto:jstreets@uci.edu}{jstreets@uci.edu}}
\author{Yury Ustinovskiy}
\address{Chandler Ullman Hall\\ Lehigh University, Bethlehem, PA, 18015}
\email{\href{mailto:yuu221@lehigh.edu}{yuu221@lehigh.edu}}

%\date{November 1, 2022}

\maketitle

\section{Introduction}

The Calabi-Yau Theorem \cite{YauCC} is a central result in mathematics with applications to complex analysis, symplectic and algebraic geometry, mathematical physics, and more.  Furthermore many questions on the fundamental structure of Calabi-Yau manifolds remain open and are actively pursued to this day.  In light of the success and applicability of this theorem, it is natural to seek extensions beyond the setting of K\"ahler geometry.  Note that one key application of the Calabi-Yau theorem is in the construction of string backgrounds \cite{CHSW}.  Despite their complexity, Calabi-Yau manifolds are actually the simplest of all possible string backgrounds, with more general geometric structures appearing in the general setting.  Indeed, it was this point of view that lead to the first appearance of \emph{generalized} K\"ahler geometry \cite{GHR}.  These structures were later rediscovered, and given this name, in the context of Hitchin's generalized geometry program \cite{HitchinGCY, GualtieriThesis,GualtieriGKG}, and in the ensuing decades it has become clear that generalized K\"ahler geometry is a deeply structured extension of K\"ahler geometry with novel implications for complex, symplectic and Poisson geometry.

In this paper we motivate a natural conjecture extending the Calabi-Yau Theorem to the simplest possible setting within generalized K\"ahler geometry, namely the case where the canonical bundles of the generalized complex structures are holomorphically trivial.  This means that the generalized complex structures are defined by global closed pure spinors (cf. Definition \ref{d:htcb}), a natural generalization of the existence of a holomorphic volume form on a K\"ahler manifold.  In this setting Gualtieri defined a \emph{generalized Calabi-Yau equation} \cite{GualtieriThesis}, which was later revealed to be locally equivalent to certain nonconvex fully nonlinear PDE \cite{HLRUZGCY}.  As a natural notion of generalized K\"ahler class has now emerged \cite{BGZ,gibson2020deformation,GualtieriHamiltonian}, it is natural to conjecture existence of a solution to the generalized Calabi-Yau equation in such a deformation class.  The main conjecture (Conjecture \ref{c:CYconj}) is then that there exists a unique generalized Calabi-Yau structure in every generalized K\"ahler class, which turns out to be classically K\"ahler Calabi-Yau (cf.\,Theorem \ref{t:rigidity}). 

Yau's original proof of the Calabi conjecture \cite{YauCC} used the methods of elliptic PDE, using the reduction to a Monge-Ampere equation for the canonically associated K\"ahler potential.  Later Cao \cite{CaoKRF} gave a parabolic analogue of this proof using the K\"ahler-Ricci flow.  A key feature of generalized K\"ahler geometry is the nonlinearity of the space of such structures, which renders a global description of a generalized K\"ahler metric in terms of a single potential function not possible (cf.\,\cite{BGZ,LindstromGKpotential} for results on \emph{local} generalized K\"ahler potentials).  For this reason an elliptic PDE approach to this conjecture seems less tractable, and we turn instead to a parabolic approach through the generalized K\"ahler-Ricci flow \cite{GKRF}.  We show that the flow is naturally adapted to this conjecture, and exhibit certain estimates and monotonicity formulas which support the conjecture (Theorem \ref{t:GKRFprops}).  Our main result shows global existence and convergence of the flow in the generalized K\"ahler class of a generalized Calabi-Yau geometry (Theorem \ref{t:longcon}).

This result gives in principle a complete description of all possible GKRF flow lines in this setting, although ultimately one wants to remove the hypothesis that there exists a fixed point of the flow in the given generalized K\"ahler class.  Nonetheless there are implications for the structure of the space of generalized K\"ahler structures on certain backgrounds.  Very little is known about the structure of this space, but we conjecture that the space of canonical deformations of a fixed generalized K\"ahler structure is contractible (cf.\,Definition \ref{d:canonicalfamily}, Conjecture \ref{c:contractibility}).  Our results confirm this conjecture in the case of a Calabi-Yau background  (Corollary \ref{cor:contractibility}).  While this general statement is fairly abstract, it has more concrete implications in some special cases. For instance, we show contractibility of certain spaces of Hamiltonian symplectomorphisms on hyperK\"ahler manifolds (Corollary \ref{c:HKcontract}), and a $\del\delb$-lemma for Calabi-Yau manifolds with commuting pairs of complex structures (Corollary \ref{c:commutingddbar}).

\subsection{Generalized K\"ahler geometry}

Generalized K\"ahler structures can be equivalently described in two ways, both of which will be relevant and used interchangeably in the sequel.  In both formulations, and throughout this paper, the relevant background is a  smooth manifold $M$ together with a closed three-form which will be denoted $H_0$.  The bihermitian description \cite{GHR} consists of a quadruple $(g, b, I, J)$ of a Riemannian metric with compatible integrable complex structures $I$, $J$, satisfying
\begin{align*}
- d^c_I \gw_I = H_0 + db = d^c_J \gw_J, \qquad d H_0 = 0.
\end{align*}
Later Gualtieri \cite{GualtieriThesis} gave an equivalent description in terms of the exact Courant algebroid $E \cong TM \oplus T^*M$ equipped with the $H_0$-twisted Dorfman bracket.  Specifically a generalized K\"ahler structure can be described as a pair of integrable complex structures $(\JJ_1, \JJ_2)$ on $E$ which further satisfy
\begin{align*}
[\JJ_1, \JJ_2] = 0, \qquad - \JJ_1 \JJ_2 = \GG,
\end{align*}
where $\GG$ denotes a generalized metric.

The integrability conditions described above do not immediately suggest natural ways to deform GK structures.  However, building upon work stretching back to Joyce \cite{AGG, ApostolovGualtieri, BGZ,gibson2020deformation, GualtieriHamiltonian}, a natural class of deformations of GK structures has emerged which generalizes the standard notion of K\"ahler class.  To briefly recall these deformations we recall that every GK structure comes equipped with a Poisson tensor \cite{AGG, PontecorvoCS, HitchinPoisson}
\begin{align}\label{f:poisson}
\gs = \tfrac{1}{2} g^{-1} [I, J].
\end{align}
Loosely speaking, on the symplectic leaves of $\gs$ we deform the complex structure by pulling $J$ back by a one-parameter family of $\gs$-Poisson deformations generated by functions $f_t$, whereas on the kernel of $\sigma$ we deform the metric by $d I d f_t$.  In the case where $\sigma$ is invertible, we may set $\Omega = \gs^{-1}$ and it follows that the deformation is determined by a family of $\Omega$-Hamiltonian diffeomorphisms acting on $J$.  This reveals the fundamentally nonlinear nature of the space of generalized K\"ahler structures, a point we return to below.  Following \cite{gibson2020deformation}, where these deformations were explicitly determined in full generality, we define a \emph{generalized K\"ahler class} as an equivalence class under one-parameter deformations as described above, and we review this in \S \ref{s:bckground}.  For a given generalized K\"ahler structure $(g,b,I,J)$ we refer to this generalized K\"ahler class by the notation $[(g,I,J,b)]$.

\subsection{The generalized K\"ahler Calabi conjecture}

In light of the discussion above it is natural to seek \emph{canonical representatives of generalized K\"ahler classes}.  Previous works in mathematics and physics literature \cite{GualtieriThesis, HLRUZGCY}, have suggested a natural equation for generalized K\"ahler structures generalizing the Calabi-Yau condition.  To describe this equation we briefly recall the spinor description of generalized complex structures.  

Elements of $TM \oplus T^*M$ act naturally on differential forms via
\begin{align*}
    (X + \xi) \cdot \psi = i_X \psi + \xi \wedge \psi.
\end{align*}
A spinor $\psi$ defines a generalized complex structure $\JJ$ on $TM \oplus T^*M$ if
\begin{align*}
    \Ker (\JJ - \i \Id) = \{ X + \xi \in (TM \oplus T^*M) \otimes \mathbb C\ |\ (X + \xi) \cdot \psi = 0 \}.
\end{align*}
Locally every generalized complex structure is described in this way, where $\psi$ is a nonvanishing section of the canonical bundle of $\JJ$.  We use the notation $\JJ_{\psi}$ to indicate a generalized complex structure of this form.  Furthermore, the integrability condition for $\JJ$ is equivalent to the existence of a section $X + \xi$ such that
\begin{align*}
    d_{H_0} \psi = (X + \xi) \cdot \psi, \qquad d_{H_0} := d + H_0 \wedge.
\end{align*}
The operator $d_{H_0}$ is the twisted Hodge differential for the background three-form defining the Courant algebroid structure, and furthermore defines a twisted cohomology theory \cite{RohmWitten} (cf.\,\cite{CavalcantiBook}) which plays a key role throughout this work.  A special case occurs when the defining spinor is \emph{closed}, i.e., $d_{H_0} \psi = 0$.  Note that it is not even true locally that every generalized complex structure can be defined by a closed spinor, for instance near points where the type is discontinuous.  If there is a closed spinor defining $\JJ$ globally then we say that $\JJ$ has \emph{holomorphically trivial canonical bundle}.  We emphasize here that we will use this terminology throughout all statements below, and it always refers to the canonical bundles of the \emph{generalized} complex structures $\JJ_i$, not the canonical bundles of the underlying classical complex structures, although we show here that in fact in our setting these bundles also admit flat metrics (cf.\,Proposition \ref{p:trivialcanonical}).

Thus, if we are given a generalized K\"ahler structure $(\JJ_1, \JJ_2)$ where the $\JJ_i$ are determined by global closed pure spinors $\psi_i$, we say that it is a \emph{generalized Calabi-Yau geometry} \cite{GualtieriThesis} if
\begin{align} \label{f:gCY}
\Phi := - \log \frac{(\psi_1, \bar{\psi}_1)}{(\psi_2, \bar{\psi}_2)} \equiv \gl.
\end{align}
In the formula above $(\cdot,\cdot)$ denotes the Mukai pairing
on spinors (see Definition~\ref{d:mukai}), and $\gl$ denotes a constant which is determined a priori from cohomological data (cf.\,Definition \ref{d:topCY}).  In the K\"ahler setting, the two spinors are $\psi_1 = \bar\Theta$, where $\Theta$ is a holomorphic volume form and $\psi_2 = e^{\i \gw}$, yielding
\begin{align*}
    \Phi = \log \left(\frac{2^n}{n!\i^{n^2}} \frac{\gw^n}{\Theta\wedge\bar{\Theta}}\right),
\end{align*}
the usual Ricci potential.  Given the background above, we can now state the main conjecture (cf. also \cite{hu2013hodge}):

\begin{conj} \label{c:CYconj} Let $(M^{2n}, g, b, I, J)$ be a compact generalized K\"ahler manifold with holomorphically trivial canonical bundles.  Then there exists a unique generalized Calabi-Yau geometry $(g_{CY}, b_{CY}, I, J_{CY}) \in [(g,b,I,J)]$, and furthermore $(g_{CY}, I)$ and $(g_{CY}, J_{CY})$ are both K\"ahler Ricci-flat.
\end{conj}

\subsection{Nondegenerate perturbations and transgression formula}

The equation (\ref{f:gCY}) defining generalized Calabi-Yau geometries is quite simple and natural from the point of view of the spinor formulation of generalized K\"ahler geometry.  On the other hand, it is not immediately clear how to connect this equation to classical curvature quantities as the spinors only implicitly define the associated generalized complex structures, which are then related to the classical bihermitian data via the delicate Gualtieri map \cite{GualtieriThesis}.  We recall that a generalized K\"ahler structure comes equipped with two Bismut connections $\N^I = \N + \tfrac{1}{2} g^{-1} H$, $\N^J = \N - \tfrac{1}{2} g^{-1} H$ associated to the pluriclosed structures $(g, I)$ and $(g, J)$, where $\nabla$ is the Riemannian connection of $g$.  These have curvature tensors $\Omega^I, \Omega^J$, and Ricci curvature forms
\begin{align*}
    \rho_I = \tfrac{1}{2} \tr \Omega^I \circ I, \qquad \rho_J = \tfrac{1}{2} \tr \Omega^J \circ J.
\end{align*}
Our first results are a series of identities relating certain curvature and torsion tensors to the underlying spinors.  Key examples are the formulas
\begin{align} \label{f:introtransgression}
    \rho_I = - \tfrac{1}{2} d J d \Phi, \qquad \rho_J = - \tfrac{1}{2} d I d \Phi,
\end{align}
extending the classic trangression formula for the Ricci forms in K\"ahler geometry to generalized K\"ahler geometry.

The proof of (\ref{f:introtransgression}), and related formulas (cf.\,Proposition \ref{p:transgression}), is suprisingly subtle.  It would seem that a natural approach would be to express the spinors $\psi_i$ in terms of bihermitian data by first obtaining the generalized complex structures through the Gualtieri map, then obtaining the spinors through the Darboux theorem for generalized complex structures.  
This runs into difficulties due to the fact that in general these two spinors will be defined using complex coordinates for the two distinct complex structures, with no easy relationship between the two.  This issue does not arise however in the case that both generalized complex structures are of symplectic type, so the generalized K\"ahler structure is \emph{nondegenerate} (cf.\,Definition \ref{d:nondeg}).  Thus the strategy of the proof is to show that, generically, it is possible to locally approximate any generalized K\"ahler structure by nondegenerate structures.  Thus after establishing the relevant formulas in the nondegenerate setting, we obtain the general result.

The perturbation result mentioned above is a useful tool in general for investigating the local structure of generalized K\"ahler manifolds, and we use it to establish many other formulas here, and further applications will appear in \cite{ASUScal}.  We record here an informal statement, referring to \S \ref{s:perturbations} for the precise statement.
\begin{customthm}{\ref{thm:nondeg_perturb}}
	\it{
    Generalized K\"ahler structures are generically smoothly approximable by nondegenerate generalized K\"ahler structures.
}
\end{customthm}
\noindent The proof is a local version of Goto's stability result \cite{GotoDef} (cf.\,also \cite{CavalcantiGoto}), a generalization of the classic Kodaira-Spencer stability theorem to the setting of generalized K\"aher geometry.  To prove Theorem \ref{thm:nondeg_perturb} we adapt this proof to the local setting, which requires us to develop Hodge theory for a certain twisted differential on manifolds with boundary.  Our deformations exploit Hitchin's original key observation that generalized complex structures can smoothly interpolate between complex and symplectic structures.  In particular, as a special case we deform a holomorphic volume form locally defining a complex structure by a symplectic form, precisely as explained in \cite{HitchinGCY} \S 4.2.  Then with the elliptic PDE results we produce a deformation of the second generalized complex structure which preserves the generalized K\"ahler condition.

\subsection{K\"ahler rigidity of compact generalized Calabi-Yau geometries}

Exploiting the transgression formulas described above, we next confirm the K\"ahler rigidity and uniqueness statements of Conjecture \ref{c:CYconj}.  We emphasize that while the equation (\ref{f:gCY}) does not admit any new solutions in the compact case, the assumption of holomorphically trivial canonical bundles makes this a special case of the more general existence and uniqueness question for generalized K\"ahler-Ricci solitons, and we briefly discuss this in \S \ref{s:introgen} below.  Furthermore, in the setting of holomorphically trivial canonical bundles, the generalized K\"ahler classes described above will preserve the data $I$ and $\sigma$, as well as the twisted cohomology classes $[\psi_1]$ and $[\psi_2]$ of the underlying spinors.  It is thus natural in this setting to give a specialized definition of a generalized K\"ahler class, which fixes these parameters (cf. Definition \ref{d:spinorGKclass}), and we refer to this space as $\mathcal {GK}^{I,\sigma}_{[\psi_1],[\psi_2]}$.  We prove uniqueness in this space up to exact $b$-field transformation, and strict uniqueness in the generalized K\"ahler class defined by canonical deformations described above.

\begin{customthm}{\ref{t:rigidity}}{\it Suppose $(g, b, I, J)$ is  a generalized Calabi-Yau geometry  on a compact manifold $M$.  Then:
\begin{enumerate}
    \item Both pairs $(g, I)$ and $(g, J)$ are K\"ahler, Ricci-flat.
    \item $(g, b, I, J)$ is the unique generalized Calabi-Yau geometry in  $\mathcal{GK}^{I,\sigma}_{[\psi_1],[\psi_2]}$, up to exact $b$-field transformation.
    \item $(g,b,I,J)$ is the unique generalized Calabi-Yau geometry in $[(g,b,I,J)]$.
\end{enumerate}
}
\end{customthm}
\noindent The K\"ahler rigidity has been previously observed in \cite[Proposition 2.6]{HuHuangHKC}.  For convenience we give a different proof here, hinging on the transgression formulas discussed above.  Applying maximum principle arguments to these identities gives the K\"ahler rigidity.  To obtain the uniqueness we first need the possible spinorial descriptions of a generalized K\"ahler structure which is actually K\"ahler, Calabi-Yau.  We achieve this in Proposition \ref{p:KCYspinors} using the Bochner method and the Beauville-Bogomolov decomposition theorem ~\cite{Beauville,bogomolov}.  In particular it will follow that the twisted deRham clases $[\psi_i]$ uniquely determine the underlying K\"ahler classes, finishing the proof of uniqueness.

\subsection{Generalized K\"ahler-Ricci flow and the Calabi conjecture}\label{flowsec}

In \cite{GKRF} a natural tool was introduced for constructing canonical metrics in generalized K\"ahler geometry, the \emph{generalized K\"ahler-Ricci flow} (GKRF).  A one-parameter family $(g_t, b_t, I_t, J_t)$ satisfies GKRF if and only if
\begin{gather} \label{f:GKRFintro}
\begin{split}
\dt g =&\ - 2 \Rc + \frac{1}{2} H^2, \qquad \dt b = - d^*_g H, \qquad H = H_0 + db,\\
\dt I =&\ L_{\theta_I^{\sharp}} I, \qquad \dt J = L_{\theta_J^{\sharp}} J,
\end{split}
\end{gather}
where $H^2(X, Y) = \IP{ i_X H, i_Y H}$, and $\theta_I = I d^*_g \gw_I, \theta_J = J d^*_g \gw_J$ are the Lee forms.  This system of equations is a special case of pluriclosed flow \cite{PCF}, which itself is a special case of generalized Ricci flow \cite{GRFbook}.  By pulling back by an appropriate family of diffeomorphisms, it is possible to fix the complex structure $I$, and we will make this choice throughout (cf.\,\S \ref{s:GKRFbackground}).  Our next main theorem gives a series of qualitative properties and universal estimates for GKRF in the setting of holomorphically trivial canonical bundles.  In the statement below the functions $\Psi_i = - \log \frac{(\psi_i, \bar{\psi}_i)}{dV_g}$ are partial Ricci potentials, and the scalar curvatures $R^{H,\Psi_i}$ are defined in \S \ref{ss:scalar}.

\begin{customthm} {\ref{t:GKRFprops}} {\it Let $(M^{2n}, g, b, I, J)$ be a compact generalized K\"ahler manifold with holomorphically trivial canonical bundles.  Let $(g_t, b_t, I, J_t)$ be the solution to generalized K\"ahler-Ricci flow with this initial data.  The following hold:
\begin{enumerate}
    \item The canonical bundles of $\JJ_i^t$ are holomorphically trivial for all time, defined by closed pure spinors $\psi_1^t \in [\psi_1^0] = \ga, \psi_2^t \in [\psi_2^0] = \gb$.  Furthermore
    \begin{align*}
    (g_t, b_t, I, J_t) \in [(g_0,b_0,I,J_0)] \subset \mathcal{GK}^{I,\sigma}_{\ga,\gb}.
    \end{align*}
    \item For any smooth existence time $t$ one has Ricci potential bounds
    \begin{align*}
        \sup_{M \times \{t\}} \left( \Phi^2 + t \brs{\N \Phi}^2 \right) \leq \sup_{M \times \{0\}} \Phi^2.
    \end{align*}
    \item For any smooth existence time $t$ one has scalar curvature bounds
    \begin{align*}
        \inf_{M \times \{0\}} R^{H,\Psi_1} \leq&\ R^{H,\Psi_1} (\cdot, t) \leq - \inf_{M \times \{0\}} R^{H,\Psi_2},\\
        \inf_{M \times \{0\}} R^{H,\Psi_2} \leq&\ R^{H,\Psi_2} (\cdot, t) \leq - \inf_{M \times \{0\}} R^{H,\Psi_1}.
    \end{align*}
    \item There exist Mabuchi-type functionals $\mathcal M_i := \int_M \Phi (\psi_i, \bar{\psi_i})$ whose only critical points are generalized Calabi-Yau geometries, and which are bounded and monotone along GKRF.
\end{enumerate}
}
\end{customthm}

Theorem \ref{t:GKRFprops} shows that the GKRF is naturally adapted to address Conjecture \ref{c:CYconj}.  The proof relies in a delicate way on key properties of the Ricci potential $\Phi$ and partial Ricci potentials $\Psi_i$.  In particular, we show that $\Phi$ satisfies the time-dependent heat equation along GKRF, leading to item (2).  Surprisingly, we show that the partial Ricci potentials $\Psi_i$ are both solutions of the \emph{dilaton flow} as defined in \cite{Streetsscalar}, which arises naturally from the renormalization group flow description of generalized Ricci flow.  As shown in \cite[Proposition 1.1]{Streetsscalar} these dilaton flows define natural weighted scalar curvature quantities $R^{H,\Psi_i}$ which are a priori bounded below, leading to item (3).  The functionals $\mathcal M_i$ reduce up to sign to the classic Mabuchi functional \cite{Mabuchi} in the case of K\"ahler metrics, and are related to the scalar curvature, suitably interpreted as a moment map for a natural symmetry action on $TM \oplus T^*M$ (\cite{goto-21, ASUScal}).  The bound for $\mathcal M_i$ is determined by $[\psi_i]$, and is conjecturally sharp.

Based on the formal properties of the flow from Theorem \ref{t:GKRFprops}, we make the following natural refinement of Conjecture \ref{c:CYconj}:

\begin{conj} \label{c:CYconjGKRF} Let $(M^{2n}, g, b, I, J)$ be a compact generalized K\"ahler manifold with holomorphically trivial canonical bundles.  Then the generalized K\"ahler-Ricci flow with this initial data preserves the generalized K\"ahler class, exists for all time, and converges to the unique generalized Calabi-Yau geometry $(g_{\infty}, b_{\infty}, I, J_{\infty})$ in this class, where $(g_{\infty}, I)$ and $(g_{\infty}, J_{\infty})$ are both K\"ahler Ricci-flat.
\end{conj}

Our next result confirms the global existence portion of Conjecture \ref{c:CYconjGKRF} under the further hypothesis that the manifold $(M, I)$ is K\"ahler.  Also we show in this setting that the Mabuchi energies converge to their expected topological bounds.  If we furthermore assume that our initial data lies in the generalized K\"ahler class of a fixed point, we prove the convergence.

\begin{customthm}{\ref{t:longcon}}{\it Let $(M^{2n}, g, b, I, J)$ be a compact generalized K\"ahler manifold with holomorphically trivial canonical bundles.  
\begin{enumerate}
    \item Suppose $(M, I)$ is a K\"ahler manifold.  Then the solution to generalized K\"ahler-Ricci flow with initial data $(g, b, I, J)$ exists for all time, and the Mabuchi energies converge to their topologically determined extreme values.
    \item Suppose there exists a generalized Calabi-Yau geometry in $[(g, b, I, J)]$.  Then the solution to generalized K\"ahler-Ricci flow with initial data $(g, b, I, J)$ converges exponentially to this necessarily unique generalized Calabi-Yau geometry.
\end{enumerate}
}
\end{customthm}

\begin{rmk} Both extra hypotheses, namely that $(M, I)$ is K\"ahler, and that there exists a Calabi-Yau geometry in the generalized K\"ahler class of the initial data, would follow a fortiori from Conjecture \ref{c:CYconj}.  Thus these hypotheses are natural for the problem, although ultimately one aims to remove them.  In fact, given Conjecture \ref{c:CYconj}, Theorem \ref{t:longcon} describes all possible solutions to GKRF with initial data having holomorphically trivial canonical bundles.  We note that Theorem \ref{t:longcon} extends the main results of \cite{ASnondeg} in a few ways.  First, the work \cite{ASnondeg} addressed the specific case of nondegenerate GK structures, whereas here we assume the more general hypothesis of holomorphically trivial canonical bundles.  Also, in \cite{ASnondeg} some weak convergence results were obtained, whereas here in full generality we obtain smooth convergence within a generalized K\"ahler class, with new implications for the nondegenerate setting (cf. Corollary \ref{c:HKcontract}).
\end{rmk}

The proof of global existence relies on a series of delicate estimates coming from a certain $1$-form/scalar reduction for the solution to pluriclosed flow underlying GKRF in this setting, modeled after the strategy of \cite{ASnondeg}.  The hypothesis of a K\"ahler background removes certain background terms from these equations, yielding important a priori estimates.  Combined with the favorable evolution equations for the Ricci potential and its gradient, together with the transgression formula for Ricci curvature, we are able to establish uniform parabolicity of the flow.  The higher regularity then follows from general results for pluriclosed flow \cite{StreetsPCFBI, JordanStreets,JFS}.  The proof of convergence relies on a different strategy, inspired by \cite{WangCF}.  In particular, we connect a K\"ahler Calabi-Yau structure to some given initial data by a one-parameter family of GK structures.  Then by the continuity method we claim that the solution to GKRF with initial data chosen anywhere along this path will converge to the K\"ahler Calabi-Yau structure.  The openness follows from continuous dependence of the flow on its initial data, together with a general stability result for Hermitian curvature flows \cite{HCF}.  The closedness is more delicate, and relies crucially on the fact that the global existence has already been established unconditionally, as well as the fact that there is universal a priori decay of the gradient of the Ricci potential from Theorem \ref{t:GKRFprops}.  These estimates allow us to show a certain backwards regularity property for GKRF in this setting.  This in turn allows us to show that the solution must enter a small ball around a Calabi-Yau structure in a bounded amount of time, after which the stability result shows the convergence.

\subsection{Structure of the space of generalized K\"ahler metrics}

As discussed above, the topology of the space of generalized K\"ahler structures is potentially nontrivial, with no obvious linear structure.  We also note that, while the space of K\"ahler metrics compatible with a fixed complex structure is naturally a cone, and thus contractible, the overall space of all K\"ahler pairs $(g, J)$ on a given smooth manifold can be quite complicated, even having infinitely many connected components.  Thus it is natural to fix certain parameters as part of the background when discussing the space of generalized K\"ahler metrics.  Following the lead of many works \cite{AGG, BGZ, GualtieriBranes, GualtieriHamiltonian}, we consider all generalized K\"ahler structures compatible with the structure of a \emph{holomorphic Poisson manifold}.  In particular, given $(M^{2n}, I, \sigma)$ a holomorphic Poisson manifold, we let %$\mathcal {GK}$ denote the space of all generalized K\"ahler structures on $M$ and let
\begin{align*}
    \mathcal{GK}^{I,\sigma} = \left\{ (g,b, I, J)\ \mbox{generalized K\"ahler}\ |\ \tfrac{1}{2} g^{-1} [I, J] = \sigma \right\}.
\end{align*}

Given an element $m = (g,b,I,J) \in \mathcal{GK}^{I,\sigma}$, there is a natural class of deformations within $\mathcal {GK}^{I,\sigma}$, dubbed \emph{canonical deformations}, defined using families of closed two-forms (\cite{gibson2020deformation}, cf.\,Definition \ref{d:canonicalfamily} below).  The equivalence class under such deformations, denoted $\mathcal{GK}^{I,\sigma}(m)$, is a natural generalization of space of K\"ahler metrics on a fixed complex manifold to the GK setting.  Within this space are equivalence classes of \emph{exact canonical deformations}, which are the generalized K\"ahler classes described above.  We conjecture that $\mathcal{GK}^{I,\sigma}(m)$ is contractible:

\begin{conj} \label{c:contractibility} Let $(M^{2n}, I, \sigma)$ be a holomorphic Poisson manifold.  Given $m = (g,b,I,J) \in \mathcal {GK}^{I,\sigma}$, then
\begin{align*}
\mathcal{GK}^{I,\sigma}(m) \cong \star.
\end{align*}
\end{conj}

A mild extension of Theorem \ref{t:longcon} implies Conjecture \ref{c:contractibility} in the case of a K\"ahler Calabi-Yau structure:

\begin{customcor}{\ref{cor:contractibility}}{\it Let $(M^{2n}, g, I)$ be a K\"ahler Calabi-Yau manifold which is part of a generalized Calabi-Yau geometry $m = (g,b, I, J)$. Then
\begin{align*}
\mathcal{GK}^{I,\sigma}(m) \cong \star.
\end{align*} 
}
\end{customcor}

The result above has interesting and more classical interpretations in the extreme cases when $\gs$ is invertible and $\gs$ vanishes.  We first address the case when $\gs$ is invertible, which includes the setting of hyperK\"ahler backgrounds.  In particular, fixing a K\"ahler form $\gw_K$ coming from a hyperK\"ahler structure $(g, I, J, K)$, we show that the connected component of the identity in the space of \emph{positive} $\gw_K$-Hamiltonian symplectomorphisms is contractible:

\begin{customcor} {\ref{c:HKcontract}}
{\it
Let $(M^{4n}, g, I, J, K)$ be a hyperK\"ahler manifold.  Let
\begin{align*}
    \Ham^+(\gw_K) := \{ \phi \in \Ham(\gw_K)\ |\ \phi^*\gw_I(X,IX)>0\ \mbox{for nonzero }X\in TM\}.
\end{align*}
Then the connected component of the identity $\Ham^+_0(\gw_K)\subset \Ham^+(\gw_K)$ is contractible:
\begin{align*}
    \Ham^+_0(\gw_K) \cong \star.
\end{align*}
Furthermore, $\Ham^+_0(\gw_K)\cap \mathrm{Aut}(M,J)=\{\mathrm{id}\}$.
}
\end{customcor}
\noindent Both claims of the Corollary follow from the fact that GKRF provides a canonical homotopy flowing $\phi \in \Ham_0^+(\gw_K)$ to the identity map.

We now turn to the case $\gs = 0$.  Here there is a holomorphic splitting of $T^{1,0}_IM$ according to the eigenspaces of $Q = IJ$ \cite{ApostolovGualtieri}.  In this setting it follows that canonical families furthermore preserve $J$, and thus the splitting.  It then follows from the integrability conditions \cite{ApostolovGualtieri} that a positive linear combination of generalized K\"ahler metrics is again a generalized K\"ahler metric, and so Conjecture \ref{c:contractibility} holds in this setting.  Using this we give a complete picture of the existence and convergence of GKRF with vanishing Poisson tensor on K\"ahler Calabi-Yau manifolds.  This result recaptures Cao's result \cite{CaoKRF} on the global existence and convergence of the K\"ahler-Ricci flow when $c_1 = 0$.

\begin{customcor}{\ref{c:commutingcase}} {\it
Let $(M^{2n}, g, b, I, J)$ be a compact generalized K\"ahler manifold satisfying
\begin{enumerate}
    \item $\gs = 0$,
    \item $(M, I)$ is K\"ahler and $c_1(M, I)=0$.
\end{enumerate}
Then the solution to generalized K\"ahler-Ricci flow with initial data $(g, b, I, J)$ exists for all time and converges to a K\"ahler Calabi-Yau metric.
}
\end{customcor}

As a further application of Corollary \ref{c:commutingcase} we obtain a global $\partial\bar\partial$-lemma for generalized K\"ahler structures on $M$ compatible with fixed commuting complex structures $I, J$ on $M$, and with an Aeppli cohomology class in $H^{1,1}_{A}(M, I)$.  The precise statement is:

\begin{customcor} {\ref{c:commutingddbar}}
{\it Let $(M^{2n}, I)$ be a compact K\"ahler manifold with $c_1(M, I) = 0$ and suppose its tangent bundle  admits a holomorphic and involutive splitting $TM= T_+ \oplus T_-$.  Let  $J$ be the integrable complex structure which equals $I$ on $T_+$ and $-I$ on $T_-$.  Suppose $(g, b, I, J)$ and $(g',b',I,J)$ are generalized K\"ahler structures such that $[\gw_I]_A = [\gw'_I]_A$.  Then there exists $\phi \in C^{\infty}(M)$ such that
\begin{align*}
    \gw_I =&\ \gw_I' + \pi^{1,1}_I d J d \phi = \gw_I' + \i \left( \del_+ \delb_+ - \del_- \delb_- \right) \phi,
\end{align*}
where $\del_{\pm}, \delb_{\pm}$ are defined in terms of the $I$-invariant splitting  $TM=T_+ \oplus T_-$.}
%{\itLet $(M^{2n}, I)$ be a compact K\"ahler manifold with $c_1(M, I) = 0$.  Further suppose $J$ is an integrable complex structure such that $[I, J] = 0$.  Suppose $(g, b, I, J)$ and $(g',b',I,J)$ are generalized K\"ahler structures such that $[\gw_I]_A = [\gw'_I]_A$.  Then there exists $\phi \in C^{\infty}(M)$ such that \begin{align*} \gw_I =&\ \gw_I' + \pi^{1,1}_I d J d \phi = \gw_I' + \i \left( \del_+ \delb_+ - \del_- \delb_- \right) \phi. \end{align*}}
\end{customcor}

\noindent A general $\del\delb$-lemma for generalized K\"ahler structures was proved in \cite{GualtieriHodge} using Hodge theory for the twisted differential and generalized K\"ahler identities.  This result is phrased in terms of the underlying generalized complex structures, and is distinct from the result of Corollary \ref{c:commutingddbar}.

\subsection{Remarks on the general case} \label{s:introgen}

We emphasize that while the case of generalized K\"ahler structures admitting global closed pure spinors is of interest, and has nontrivial applications to understanding the structure of Calabi-Yau manifolds as described above, it should be considered a special case of the much more subtle existence and uniqueness question for compact generalized K\"ahler-Ricci solitons (GKRS).  Recently, in a series of papers due to the authors \cite{Streetssolitons, SU1, ASU3} we gave a complete classification of compact GKRS on complex surfaces.  Roughly speaking the proof consists of using results from the Kodaira classification and results on the structure of GK manifolds to show that non-K\"ahler GKRS can only exist on type 1 (diagonal) Hopf surfaces, and their finite quotients.  We construct solitons on these manifolds by imposing maximal symmetry \cite{Streetssolitons, SU1}, then show the uniqueness by establishing that maximal symmetry must hold a priori via variational arguments using a generalization of Aubin's $I$-functional to this setting \cite{ASU3}.  While this represents a complete answer in this dimension, it remains unclear how to formulate a plausible general conjecture for the existence and uniqueness of GKRS in higher dimensions.

\vskip 0.1in
\textbf{Acknowledgements:} The authors would like to thank Gil Cavalcanti, Vicente Cortes, Marco Gualtieri and Michael Taylor for helpful conversations.  The first named author was supported by an NSERC Discovery Grant and a Connect Talent fellowship of the Region des Pays de la Loire.  The second named author was  supported by National Key R\&D Program of China 2023YFA1009900. The third named author was supported by a Simons Fellowship and by the NSF via DMS-1454854 and DMS-2203536.

\section{Background on generalized K\"ahler geometry} \label{s:bckground}
\subsection{Fundamental definitions}

We recall here some basic definitions in generalized complex and K\"ahler geometry.  Our discussion is brief, and we refer to \cite{GualtieriGCG} for further detail.  The basic object  underpinning these constructions is an \emph{exact Courant algebroid}, which by a result of Severa \cite{Severa} is described by the following setup: fix $M$ a smooth manifold and let $E = TM \oplus T^*M$, equipped with the neutral inner product
\begin{align*}
    \IP{X + \xi, Y + \eta} = \tfrac{1}{2} \left( \xi(Y) + \eta(X) \right).
\end{align*}
Furthermore we endow this vector bundle with a (twisted) Dorfman bracket:
\begin{align*}
    [X + \xi, Y + \eta] = [X,Y] + L_X \eta - i_Y d \xi + i_Y i_X H_0,
\end{align*}
where $H_0$ is a closed three-form.  We note that the Dorfman bracket extends $\mathbb C$-linearly to the complexified bundle $(TM \oplus T^*M) \otimes \mathbb C$.  These structures are part of the background and will not be emphasized in the sequel.

\subsection{Generalized complex structures} \label{ss:GCS}

\begin{defn}
A \textit{generalized almost complex structure} on $(M,H_0)$ is a $\la{\cdot,\cdot}$-orthogonal endomorphism $\J\in\End(TM\oplus T^*M)$ such that $\J^2=-\mathrm{Id}$. A generalized almost complex structure $\J$ is \textit{integrable}, if its $\i$-eigenspace $L^{1,0}\subset (TM\oplus T^*M)\otimes\C$ is closed under the Dorfman bracket $[\cdot,\cdot]_{H_0}$. In this case we will call $\J$ a \textit{generalized complex structure}.
\end{defn}

Every generalized complex structure $\J$ defines a decomposition of the complexified generalized tangent bundle $(TM\oplus T^*M)\otimes\C$ into the direct sum of maximal isotropic $\pm\i$ eigenspaces:
\[
(TM\oplus T^*M)\otimes\C=L^{1,0}\oplus L^{0,1}.
\]
Similarly to the classical case, the integrability of $\J$ is equivalent to the vanishing of the  Nijenhuis tensor $N_{\J}\in \mathrm{Hom}(\Lambda^2(L^{1,0}),L^{0,1})$ defined by
\[
N_\J(e_1,e_2):=\pi_{L^{0,1}}([e_1,e_2]_{H_0}).
\]

Elements of $TM\oplus T^*M$ naturally act on the differential forms: $(X+\xi)\cdot \alpha=i_X\alpha+\xi\wedge \alpha$ for $X+\xi\in TM\oplus T^*M$ and $\alpha\in \Lambda^*(T^*M)$. For any $\varphi\in \Lambda^*(T^*M)_\C$, its kernel $\Ker(\varphi)\subset (TM\oplus T^*M)\otimes\C$ is an isotropic subspace. At a given point, the \emph{maximal} isotropic subspace $L^{1,0}\subset (TM\oplus T^*M)\otimes\C$ is defined by a \emph{pure} spinor $\varphi\in\Lambda^*(T^*M)_\C$:
\[
L^{1,0}=\Ker(\varphi)
\]
and this pure spinor is uniquely defined up to a non-zero complex multiple. In this case we will often write $\J=\J_{\varphi}$, implying that the generalized complex structure $\J_{\varphi}$ is determined by a spinor $\varphi\in\Lambda^*(T^*M)_\C$.

\begin{defn}\label{d:mukai}
Given $\varphi,\psi\in\Lambda^*(T^*M)_\C$, we define \textit{Mukai pairing} as
\begin{equation}\label{f:mukai}
    (\varphi,\psi):=(2\i)^{-n}[\varphi\wedge s(\psi)]_{\mathrm{top}},
\end{equation}
where $[\alpha]_{\mathrm{top}}$ denotes the top $2n$-degree component of a differential form $\alpha\in\Lambda^*(T^*M)_\C$, and $s(\alpha)$ is the Clifford involution defined on the decomposables as
\[
s(dx^1\wedge\dots \wedge dx^k)=dx^k\wedge\dots\wedge dx^1.
\]
\end{defn}
Note for a pure spinor $\phi$, since $L^{0,1}=\Ker(\bar{\varphi})$ and $L^{1,0}\cap L^{0,1}=\{0\}$, by \cite{Chevalley} we have
\begin{equation}\label{f:chevalley}
(\varphi,\bar{\varphi})\neq 0.
\end{equation}
With our normalization convention, for a closed pure spinor $\varphi$ defining a generalized complex structure, the pairing $(\varphi,\bar{\varphi})$ is a \emph{real} volume form.

Let us review two key examples of generalized complex structures on an \emph{untwisted background} $(M,H_0=0)$:
\begin{ex}[Generalized complex structure of symplectic type]\label{ex:gc_symp}
Let $\omega$ be a symplectic form on $M$. Then 
\[
\J:=\left(\begin{matrix}0 & -\omega^{-1}\\ \omega & 0\end{matrix}\right)
\]
is an integrable generalized complex structure on $(M,H_0=0)$. Its $\i$-eigenspace is
\[
L^{1,0}=\{X-\i\omega(X,\cdot)\ |\ X\in TM\},
\]
and there is a pure spinor
\[
\psi=e^{\i\omega}.
\]
In accordance with~\eqref{f:chevalley} we have $(e^{\i\omega},e^{-\i\omega})=\frac{\omega^n}{n!}\neq 0$.
\end{ex}
\begin{ex}[Generalized complex structure of complex type]\label{ex:gc_cpx}
Let $J$ be a usual integrable complex structure on $M$. Then
\[
\J:=\left(\begin{matrix}J & 0\\ 0 & -J^*\end{matrix}\right)
\]
is an integrable generalized complex structure on $(M,H_0=0)$. Its $\i$-eigenspace is
\[
L^{1,0}=T^{1,0}M\oplus \Lambda^{0,1}(M).
\]
If $\Theta$ is a local non-vanishing section of the canonical bundle $K_J=\Lambda^{n,0}(M)$, then there is a (locally defined) pure spinor
\[
\psi=\bar\Theta.
\]
If $\Theta=dz^1\wedge\dots\wedge dz^n$, where $z^i=x^i+\i y^i$ are local holomorphic coordinates, then $(\bar\Theta,\Theta)=dx^1\wedge dy^1\wedge\dots\wedge dx^n\wedge dy^n$.
\end{ex}

The collection of all the multiples of the pure spinors defined by $L^{1,0}$ over $x_0\in M$ determines a complex line bundle $K_\J\subset \Lambda^*(T^*M)_\C$ which is called \textit{the canonical line bundle of~$\J$}. The integrabilty condition for $\J$ can be read off from $K_\J$ as follows. Let $\varphi$ be any local smooth section of $K_\J$ such that $(\varphi,\bar\varphi)\neq 0$. Then there exists $e\in (TM\oplus T^*M)\otimes \C^*$ such that
\begin{equation}\label{f:spinor_integrability}
d_{H_0}\varphi=e\cdot \varphi.
\end{equation}

\begin{defn}[$b$-field transform]
Let $b\in \Lambda^2(T^*M)$ be a 2-form. A $b$-field transform is an operator
\[
e^b\colon TM\oplus T^*M\to TM\oplus T^*M,\quad e^b(X+\xi)=X+\xi+b(X,\cdot).
\]
This endomorphism preserves the $H_0$-twisted Dorfman  bracket if and only if $db=0$. In general, it transforms the $H_0$-twisted Dorfman bracket into the $(H_0+db)$-twisted Dorfman bracket:
\[
e^{-b}[e^b(X+\xi),e^b(Y+\eta)]_{H_0}=[X+\xi,Y+\eta]_{H_0+db}.
\]
\end{defn}

Note that $b$-field transformations act on the space of generalized complex structures.  Specifically, if $\J$ is an integrable generalized complex structure on $(M,H_0)$, then $e^{-b}\J e^b$ is an integrable generalized complex structure on $(M,H_0+db)$, in particular, once the 2-form $b$ is closed, $e^{-b}\J e^b$ is another generalized complex structure on the same background. If $\J$ has the $\i$-eigenspace $L^{1,0}$ and a local pure spinor $\varphi$, then $e^{-b}\J e^b$ has the eigenspace $e^{-b}L^{1,0}$ and a local pure spinor $e^b\wedge \varphi$.

\begin{defn}
Let $\J$ be a generalized complex structure on $M^{2n}$. The \textit{type} of $\J$ at a given point $x\in M$ is
\[
\mathrm{type}(\J)=n-\frac{1}{2}\mathrm{rank}(Q),
\]
where $Q:=\pi_T(\J|_{T^*M})\in \Lambda^2(TM)$. Equivalently, $\mathrm{type}(\J)$ is the lowest degree of a pure spinor defining $\J$ at $x\in M$. The type of $\J$ is an upper-semicontinuous function on $M$. The \textit{parity} of $\J$ is the parity of $\mathrm{type}(\J)$ and does not depend on the choice of $x\in M$.  Depending on the parity, a defining spinor will lie in the relevant subspace $\Lambda^{\mbox{\tiny{even/odd}}}(T^*M)_{\C}$.
\end{defn}

Whenever $\mathrm{type}(\J)$ is locally constant in a neighbourhood of $x\in M$, we can find a Darboux coordinate system and a $b$-field transform putting $\J$ into a canonical form.  This result is key to our perturbation results below.

\begin{thm}[{Normal form of a generalized complex structure, \cite[Thm.\,4.35]{GualtieriThesis}}]\label{thm:darboux}
    Let $\J$ be a generalized complex structure on $(M,H_0)$. Assume that $\mathrm{type}(\J)=k$ is constant in a neighbourhood of $x\in M$. Then there exists $b\in \Lambda^2(T^*M)$ with $db=H_0$ and a diffeomorphism
    \[
    f\colon U\to V
    \]
    between a neighbourhood of $x$ and an open subset $V\subset \C^k\times \R^{2n-2k}$ such that
    \[
    \J\big|_{U}=e^{-b}(f^*(\J_{\mathrm{cpx}}\oplus \J_{\mathrm{symp}}))e^{b},
    \]
    where $\J_{\mathrm{cpx}}$ is the generalized complex structure on $\C^k$ given by the usual complex structure, and $\J_{\mathrm{symp}}$ is the generalized complex structure on $\R^{2n-2k}$ given by the standard symplectic form.
\end{thm}

\begin{rmk} As in Examples~\ref{ex:gc_symp} and~\ref{ex:gc_cpx}, if $k = 0$ in Theorem \ref{thm:darboux}, then such $\J$ are referred to as \emph{symplectic type} (near $x$).  Analogously if $k = n$ then such $\J$ are referred to as \emph{complex type}.
\end{rmk}

\subsection{Generalized K\"ahler structures}

\begin{defn}
A \textit{generalized K\"ahler structure} on $M$ is a pair of commuting generalized complex structures $(\J_1,\J_2)$ such that the bilinear form
\[
\la{-\J_1\J_2\cdot,\cdot}
\]
on $T\oplus T^*$ is positive definite.  We denote the operator $-\J_1\J_2$ by $\GG$, which is a \emph{generalized metric} (cf.\,\cite{GualtieriThesis,GRFbook}).
\end{defn}

The fundamental result of Gualtieri states that a generalized K\"ahler structure $(\J_1,\J_2)$ on $(M,H_0)$ can be equivalently described by certain bihermitian data on $(M,H_0)$.

\begin{thm}[{\cite[Ch. 6]{GualtieriThesis}}]\label{thm:gualtieri_map}
    Let $(\J_1,\J_2)$ be a generalized K\"ahler structure on $(M,H_0)$. Then there exist a unique bihermitian structure $(g, b, I,J)$ such that
    \begin{equation}\label{f:gualtieri_map}
    \J_{1/2}=\frac{1}{2}e^b\left(
        \begin{matrix}I\pm J & -(\gw_I^{-1}\mp \gw_J^{-1})\\ \gw_I\mp\gw_J & - (I^*\pm J^*)\end{matrix}
    \right)e^{-b},
    \end{equation}
    where the $2$-forms $\gw_I=gI$ and $\gw_J=gJ$ satisfy
    \begin{equation}\label{f:gk_torsion}
    -d^c_I\gw_I=H_0+db=d^c_J\gw_J.
    \end{equation}
    Conversely, if the data $(g,b,I,J)$ satisfies~\eqref{f:gk_torsion}, then~\eqref{f:gualtieri_map} determines a generalized K\"ahler structure on $(M, H_0)$.
\end{thm}

This work concerns a special class of generalized K\"ahler structures which are globally described by spinors:

\begin{defn} \label{d:htcb} We say that a generalized K\"ahler structure $(\JJ_1, \JJ_2)$ has \emph{holomorphically trivial canonical bundles} if there exist  closed pure spinors $\psi_i$ such that $\JJ_i = \JJ_{\psi_i}$.  In this setting we will always use the notation $\psi_i$ for the spinors without explicit mention.
\end{defn}

\begin{rmk}
 The hypothesis of having holomorphically trivial canonical bundles is strictly stronger than assuming that the canonical bundles of $\JJ_i$ are topologically trivial, evidenced by generalized K\"ahler structures on the standard (primary) Hopf surface.  Since the second cohomology with integer coefficients of the Hopf surface vanishes, the canonical bundles of $\JJ_i$ are topologically trivial. On the other hand the associated real Poisson tensor has degeneracy along two elliptic curves, and here there is type jumping of the generalized complex structures. In the holomorphic coordinates $(z,w)$ in a neighborhood of a point on an elliptic curve $\{z=0\}$ an underlying spinor is $z + dz\wedge dw$, which is not equivalent to a closed spinor.  A priori necessary and sufficient conditions for a generalized complex structure to have  holomorphically trivial canonical bundle are given in (\cite{heluani2010generalized} Proposition 3.7)
\end{rmk}

The first example of such structures are K\"ahler metrics on Calabi-Yau backgrounds, i.e., when $I=J$ and $(M, I)$ admits a holomorphic volume form $\Theta \in H^0(M, \Lambda^{n, 0}(T^*M_{\C})$.  It follows that the closed pure spinors 
\begin{equation}\label{f:CY-spinor}
\psi_1 = \bar{\Theta}, \qquad \psi_2=e^{\i \gw_I},\end{equation}
where $\gw_I$ is the K\"ahler form of $(g, I)$, determines $(g, I, J)$ in this case.  Also observe that the pair $(\psi_2, \psi_1)$  determines the  generalized K\"ahler structure $(g, I, J=-I)$.  A further key class of examples of these structures are the nondegenerate GK manifolds:

\begin{defn} \label{d:nondeg} A generalized K\"ahler structure $(\J_1, \J_2)$ is \emph{nondegenerate} if $\J_i$ are both symplectic type.  Equivalently $I \pm J$ are invertible, and equivalently the Poisson tensor $\gs=\tfrac{1}{2} g^{-1} [I, J]$ is nondegenerate.
\end{defn}

For a nondegenerate structure the endomorphisms $I \pm J\in\End(TM)$ are invertible, and there are canonically associated symplectic forms
\begin{align*}
    F_{\pm} = -2 g (I \pm J)^{-1}.
\end{align*}
If $\J_{1/2}$ are given by~\eqref{f:gualtieri_map}, then setting $b'=-b+g(I+J)^{-1}(I-J)$ we observe that
\[
\J_1=e^{-b'-2\Omega}
\left(\begin{matrix}0 & -F_-^{-1} \\ F_- & 0\end{matrix}\right)
e^{b'+2\Omega},\qquad 
\J_2=e^{-b'}\left(\begin{matrix}0 & -F_+^{-1} \\ F_+ & 0\end{matrix}\right)e^{b'}
\]
where $\Omega=2g(I+J)^{-1}(I-J)^{-1}$ is a closed 2-form, see~\cite[\S 3.1]{ASnondeg}. It furthermore follows that the underlying $d_{H_0}$-closed spinors are
\begin{equation}\label{f:nondege} \psi_1= e^{b + 2\Omega + \i F_-}, \qquad \psi_2=e^{b + \i F_+}.
\end{equation}
Thus nondegenerate generalized K\"ahler structures have holomorphically trivial canonical bundles.

\subsection{Spinor formulation of K\"ahler Calabi-Yau structures} 

As discussed in the introduction, we expect all GK manifolds with holomorphically trivial canonical bundles to be naturally deformable to a K\"ahler Calabi-Yau structure.  Thus it is important to our discussion to canonically identify the possible pure spinors underlying the GK description of such a structure.  The proposition below shows that they are described as a combination of the nondegenerate and classic K\"ahler Calabi-Yau cases described in the previous subsection.

\begin{prop} \label{p:KCYspinors}Let $(M, g, b, I, J)$ be a compact generalized K\"ahler structure such that $(g, I)$ and $(g, J)$ are K\"ahler, Ricci-flat.  Then there exists a (possibly trivial) finite cover, still denoted $(M, g, b, I, J)$ which is determined by global closed pure spinors
    \begin{align*}
        \psi_1 =&\ e^b\wedge \psi_1^{\gs} \wedge \psi_1^+ \wedge \psi_1^- = e^{b+2 \Omega + \i F_-} \wedge \bar{\Theta}_+\wedge e^{\i \gw_-},\\
        \psi_2 =&\ e^b\wedge \psi_2^{\sigma} \wedge \psi_2^+ \wedge \psi_2^- = e^{b+\i F_+} \wedge e^{\i \gw_+}\wedge \bar{\Theta}_-.
    \end{align*}
\end{prop}

\begin{proof} We first observe that it follows from a standard Bochner argument that $\sigma$, $I$, and $J$ are all parallel with respect to the Levi-Civita connection of $g$.  Thus we obtain a parallel $g$-orthogonal splitting
\begin{equation}\label{f:splitting}
TM = \im \sigma \oplus \ker (I - J) \oplus \ker (I + J) =: V_{\sigma} \oplus V_+ \oplus V_-,
\end{equation}
This splitting further determines parallel $2$-forms $\Omega = \sigma^{-1}$ and $F_{\pm}$ on $\Lambda^2(V_{\sigma})$ corresponding to a nondegenerate generalized structure $(g_{\sigma}, I_{\sigma}, J_{\sigma})$ on $V_{\sigma}$, as well as to generalized K\"ahler structures of the form $(g_+, I_+, J_+=I_+)$  and $(g_-, I_-, J_-=-I_-)$ respectively defined on $V_+$ and $V_-$.  We further let $\gw_+ =g_+ I_+$ and $\gw_-=g_-I_-$ the corresponding K\"ahler forms on $V_+$ and $V_-$.  All of these objects are defined algebraically at one point and correspond to parallel tensors on $M$.  Thus we recover $(g, I, J)$ as the product of these (local) generalized K\"ahler structures.  The final step to obtain a spinorial description of $(g, I, J)$, will be to show that, after passing to a finite cover, the canonical bundles $\Lambda^{n_+, 0} V_+^*$ and $\Lambda^{n_-,0} V_-^*$ are holomorphically trivial, i.e., that there exists parallel holomorphic $n_{\pm}$-forms $\Theta_+$ and $\Theta_-$ over $M$, trivializing the respective bundles.  Given such forms, we obtain the spinors generating $(g, b=0, I, J)$ as the product of the respective spinors generating the generalized K\"ahler structures on $V_\sigma$, $V_+$ and $V_-$, i.e.,
\begin{align*}
    \psi_1 =&\ \psi_1^{\gs} \wedge \psi_1^+ \wedge \psi_1^- = e^{2 \Omega + \i F_-} \wedge \bar{\Theta}_+\wedge e^{\i \gw_-},\\
        \psi_2 =&\ \psi_2^{\sigma} \wedge \psi_2^+ \wedge \psi_2^- = e^{\i F_+} \wedge e^{\i \gw_+}\wedge \bar{\Theta}_-.
\end{align*}
After a $b$-field transform we get the claimed spinor description of $(g,b,I,J)$.

We now determine the relevant finite cover yielding the forms $\Theta_{\pm}$.  Using the Beauville-Bogomolov decomposition theorem~\cite{Beauville,bogomolov} and the fact that $I$ and $J$ are parallel, one can see that, up to a passing to a finite covering,  $(M, g, I)$ can be written as a product
\begin{equation}\label{f:splitting-BB}
(M, g, I)= (S, g_S, I_S) \times  (X, g_X, I_X) \times (Y, g_Y, I_Y) \times ({\mathbb T}^{2k}, g_0, I_0), \end{equation}
 where $(S, g_S, I_S)$ is a simply connected holomorphic symplectic manifold with symplectic structure $\Omega_S+iI_S\Omega_S$ and a Calabi-Yau metric $g_S$, $(X, g_X, I_X)$ and $(Y, g_{Y}, I_Y)$ are simply connected Calabi-Yau manifolds (possibly containing HK factors), and  $({\mathbb T}^{2k}, g_0, I_0)$ is a flat complex torus. Furthermore, $I$ and $J$ are of the form
 \[ I= I_S \oplus I_X \oplus I_Y\oplus I_0, \qquad J = J_S \oplus I_X \oplus -I_Y \oplus J_0, \]
 with $(g_S, I_S, J_S)$ being a nondegenerate generalized K\"ahler structure with Poisson tensor $\sigma_S=\Omega_S^{-1}$ and $(g_0, I_0, J_0)$ being a flat generalized K\"ahler structure on ${\mathbb T}^{2k}$ with corresponding Poisson tensor $\sigma_0$.  Note that also there is a parallel decomposition of the tangent space
 \[T\mathbb{T}^{2k} =  V_{\sigma_0} \oplus V_{+}^0 \oplus V_{-}^0, \]
 where $V_{\sigma_0}:=\im\sigma_0, \, V_{+}^0 = \ker(I_0-J_0), \, V_{-}^0 = \ker(I_0+J_0)$.  It follows that $V_{\sigma}= TS \oplus V_{\sigma_0}$, $V_+ = TX \oplus V_{+}^0$ and $V_-= TY\oplus V_{ -}^0$.  Thus the canonical bundles of $V_+$ and $V_-$ are trivialized by parallel forms, as required.
 \end{proof}

\section{Local genericity of nondegenerate GK structures} \label{s:perturbations}

Above we discussed the quite different different data $(\J_1,\J_2)$ and $(g,b,I,J)$ describing generalized K\"ahler structures.  While the correspondence between $(\J_1,\J_2)$ and $(g,b,I,J)$ given in Theorem~\ref{thm:gualtieri_map} is rather explicit, it is often quite difficult to translate various higher order differential-geometric objects presented in terms of $(\J_1,\J_2)$ into the classical bihermitian data $(g,b,I,J)$ and vice versa.  In this section we provide a toolbox allowing for such translation, Theorem \ref{thm:nondeg_perturb}, which says that any generalized K\"ahler structure with both $\J_i$ of even type can be generically locally approximated by nondegenerate structures.  The proof relies principally on a local version of Goto's stability result \cite{GotoDef} (cf.\,also \cite{CavalcantiGoto}), a generalization of the classic Kodaira-Spencer stability theorem to the setting of generalized K\"aher geometry.  

Goto's proof is similar in spirit to the original result, relying on Hodge theory for the $H_0$-twisted differential and the generalized K\"ahler identities for the associated Laplacian \cite{GualtieriHodge} (cf.\,also \cite{CavalSKT}).  We adapt this proof to the local setting, where we must use Hodge theory for manifolds with boundary.  We prove fundamental results on the Green's operator for the $H_0$-twisted Laplacian in the appendix.  After checking that the absolute boundary conditions for the twisted deRham differential are elliptic, the construction of the Green's operator follows standard lines.  Using this Green's operator for the boundary value problem, we are able to adapt the strategy of \cite{GotoDef, CavalcantiGoto} to localize the stability result.  Using this stability we are able to generically construct local perturbations which yield the approximation by nondegenerate structures.  We note that all of our constructions yield tensors of some arbitrary but fixed $C^{k,\ga}$ regularity.  The proofs of Kodaira-Spencer/Goto rely on solving a formal power series, initially with tensors of fixed $C^{k,\ga}$ regularity, and then a further elliptic problem is solved to yield analytic regularity.  As this is not necessary for our purposes we do not pursue this improved regularity.

\subsection{Hodge theory of generalized K\"ahler manifolds}

Here we recall some further results on generalized K\"ahler geometry, in particular natural extensions of the K\"ahler identities.  We refer generally to \cite{GualtieriThesis, GualtieriHodge, CavalSKT} for further background.  We first note that a generalized complex structure $\J$ induces a decomposition of the space of complex-valued differential forms via
\begin{align*}
    C^{\infty}(M,\Lambda^*(T^*M)_{\C}) = \bigoplus_{p=-n}^n U_\J^p,
\end{align*}
where $U^p$ is the $\i k$-eigenspace of the Lie algebra action of $\J$ on the space of spinors $\Lambda^*(T^*M)_\C$.  More concretely, if $\varphi\in C^\infty(M,\Lambda^*(T^*M)_{\C})$ is a local pure spinor defining $\J$, then $U^n_\J=C^\infty(M,\C)\cdot\varphi$ and
\[
U_\J^{n-k}=\Lambda^k(L^{0,1})\cdot U_\J^n.
\]

Given a generalized K\"ahler structure $(\J_1, \J_2)$, since $\J_1$ and $\J_2$ commute, we obtain a joint decomposition
\[
C^{\infty}(M,\Lambda^*(T^*M)_\C)=\bigoplus U^{p,q}, \qquad U^{p,q} = U^p_{\J_1} \cap U^q_{\J_2}.
\]
This bigrading does not respect the usual grading on $\Lambda^*(T^*M)_\C$ but respects the decomposition into odd/even forms.  According to this decomposition, the differential $d$ has 4 components of bidegree $(\pm 1,\pm 1)$:
\[
d=d^{1,1}+d^{1,-1}+d^{-1,1}+d^{-1,-1}.
\]
We adopt the notation of \cite{GualtieriHodge, CavalSKT} and denote
\begin{align*}
d^{1,1}=\delta_+,\ \ d^{1,-1}=\delta_-,\ \ d^{-1,1}=\bar{\delta_-},\ \ d^{-1,-1}=\bar{\delta_+}.
\end{align*}

A generalized K\"ahler structure $(\J_1, \J_2)$ determines a generalized metric $\GG = - \J_1 \J_2$, which can be expressed as an operator on $TM\oplus T^*M$
\[
\GG=e^b\left(\begin{matrix}0 & g^{-1}\\g & 0\end{matrix} \right)e^{-b}
\]
for 2-form $b\in\Lambda^2(T^*M)$ and a Riemannian metric $g$.
The relevant extensions of the K\"ahler identities use the adjoint operators $\delta_\pm^*:=\star_\GG\circ\delta_\pm\circ \star_\GG^{-1}$, where the linear operator $\star_\GG\colon\Lambda^*(T^*M)_\C\to\Lambda^*(T^*M)_\C$ is determined by
\[
(\alpha,\star_\GG\beta)_M=(e^b\wedge\alpha)\wedge \star_g(e^b\wedge\beta),
\]
and $\star_g$ is the usual Riemannian Hodge star. It was proved by Gualtieri in~\cite{GualtieriHodge} (see also~\cite{CavalcantiBook}) that the generalized K\"ahler structures satisfy the following identities:

\begin{prop}[Generalized K\"ahler identities] \label{p:GKidentities} Given a smooth manifold $(M, H_0$) with closed three-form and a generalized K\"ahler structure $(\J_1, \J_2)$, one has
\[
\delta_+^*=-\bar\delta_+,\quad \delta_-^*=\bar{\delta_-}.
\]
In particular the Laplace operators $ \Delta_{H_0}:=d_{H_0}d^*_{H_0}+d^*_{H_0}d_{H_0}$, $\Delta_{\delta_+}:=\delta_+\delta_+^*+\delta_+^*\delta_+$ and $\Delta_{\delta_-}:=\delta_-\delta_-^*+\delta_-^*\delta_-$ coincide up to a constant factor:
\[
\Delta_{H_0}=4\Delta_{\delta_+}=4\Delta_{\delta_+}.
\]
\end{prop}

\subsection{Local version of Goto's stability theorem}

Now let us concentrate on the case when $M=B^{2n}$ is equipped with a GK structure $(\J_1, \J_2)$ determined by even closed pure spinors $\varphi_1$ and $\varphi_2$ (in this case $n$ is necessarily even, but it is not relevant for the proof).  Let $\varphi_1^t$ be an analytic deformation of $\varphi_1$ through closed spinors.  We want to reprove Goto's stability result in this setting and find a corresponding deformation $\varphi_2^t$ of $\varphi_2$ such that $\J_{\varphi_1^t}$, $\J_{\varphi_2^t}$ is a GK structure.

\begin{thm}[Local Goto's stability]\label{thm:local_stab}
Let $(\J_1^0,\J_2^0)$ be a generalized K\"ahler structure on a ball $M=B^{2n}$ such that $\J_1^0$ and $\J_2^0$ are both even type, defined by closed pure spinors $\varphi_1^0$ and $\varphi_2^0$. Let $\varphi_1^t$ be any analytic deformation of $\varphi_1^0$ through closed pure spinors.  Then given $k \geq 2, 0 < \ga < 1$ there exists an analytic in $t$ deformation $\varphi_2^t$ of $\varphi_2$ such that $(\J_{\varphi_1^t}$, $\J_{\varphi_2^t})$ is a family of $C^{k,\ga}$ generalized K\"ahler structures for $|t|\leq \epsilon$ sufficiently small.
\end{thm}
\begin{proof} The proof follows a well-known inductive procedure for constructing the solution via power series.  In particular the proof follows along the lines of \cite{GotoDef,CavalcantiGoto} once a certain key local elliptic problem is solved.  We spell out the solution of this problem, omitting the remaining details.

Given an exact form
\[
\rho\in U^{-1,n-1}\oplus U^{1,n-1}\oplus U^{-1,n-3}\oplus U^{1,n-3}
\]
we wish to find its primitive $u$ of pure bidegree $(0,n-2)$:
\[
\rho=d_{H_0}u,\quad u\in U^{0,n-2}.
\]
The form $\rho$ has the same parity as $d_{H_0}\psi_t$, which is odd in our case.

To solve this problem, we are going to use the properties of the Green's operator $G^A$ constructed in the Appendix (cf. Theorem~\ref{thm:app_green}).  First note that in our setting of a ball $B^{2n}$, the $d_{H_0}$ cohomology $H^*_{d_{H_0}}(B^{2n})$, is isomorphic to $H^*(B^{2n})$ by an appropriate $b$-field transform with $H_0=db$, since the action of such $b$-field establishes an isomorphism $(\Lambda^*(T^*M)_\C,d)\simeq (\Lambda^*(T^*M)_\C,d_{H_0})$. In particular,
\[
\dim H^{\mathrm{odd}}_{d_{H_0}}(B^{2n})=0.
\]
We claim that there are no odd forms $u\in C^\infty(M,\Lambda^{\mathrm{odd}}(T^*M)_\C)$ satisfying $d_{H_0}u=d_{H_0}^*u=0$ and the absolute boundary conditions $(\star_\GG u)\big|_{\del B^{2n}}=(\star_\GG d_{H_0}u)\big|_{\del B^{2n}}=0$. Indeed, if $u$ is such a form, it must be $d_{H_0}$-exact: $u=d_{H_0}v$, and integrating $\int_M (u,\star_\GG u)$ by parts we conclude that $u\equiv 0$.

Hence for any 
$u\in C^{\infty}(\Lambda^{\mathrm{odd}}(T^*M)_{\C})$, we have $P_h^A(u)=0$, so that by Lemma~\ref{lm:app_hamonic_twisted}
\[
u = \Delta  G^Au.
\]
We claim that for an odd form $u^{p,q}\in U^{p,q}$ of pure bidegree, the form $G^Au^{p,q}$ must be of type $(p,q)$. Indeed $u^{p,q} = \Delta  G^Au^{p,q}$ and $\Delta =4\Delta_{\delta_+}$ preserves the bigrading of $U^{p,q}$. Thus any component of $G^Au^{p,q}$ of bidegree other than $(p,q)$ must be harmonic, while the odd harmonic forms on $B^{2n}$ satisfying the boundary conditions vanish.

Now, given an exact form $\rho$ as above, consider the form $G^A\rho$. From Lemma~\ref{lm:app_hamonic_twisted} in the Appendix we know that
\[
\rho=d_{H_0}(d_{H_0}^*G^A\rho),\qquad d_{H_0}G^A\rho=0.
\]
At this point we cannot yet claim that $d^*_{H_0}G^A\rho$ is the desired primitive of $\rho$ since it might not be of pure type. 
The form $G^A\rho$ has the same bigrading decomposition as $\rho$ hence from $d_{H_0}G^A\rho=0$ we have
\begin{equation*}
	\begin{split}
		0 =&\ \delta_+(G^A\rho^{1,n-1})\\
		0 =&\ \delta_-(G^A\rho^{1,n-3})\\
		0 =&\ \bar{\delta_+}(G^A\rho^{-1,n-3})\\
		0 =&\ \bar{\delta_-}(G^A\rho^{1,n-1})\\
		0 =&\ \delta_-(G^A\rho^{-1,n-1})+\bar\delta_-(G^A\rho^{1,n-3}) + \delta_+(G^A\rho^{-1,n-3})+\bar\delta_+(G^A\rho^{1,n-1}).
	\end{split}
\end{equation*}
Now we define
\[
u=4(\delta_-G^A\rho^{-1,n-1}+\bar\delta_-G^A\rho^{1,n-3}).
\]
Clearly $u\in U^{0,n-2}$ is of the desired pure type, and it only remains to check that $\rho=du$. We check it by each bigrading component using Proposition \ref{p:GKidentities}, e.g.,
\[
\begin{split}
(d_{H_0}u)^{1,n-3}&=\delta_-u=4(\delta_-\bar\delta_-G^A\rho^{1,n-3})=4(\delta_-\bar\delta_-+\bar\delta_-\delta_-)G^A\rho^{1,n-3}\\
&=\Delta_{d_{H_0}}G^A\rho^{1,n-3}=\rho^{1,n-3}.
\end{split}
\]
The remaining components of $d_{H_0}u$ are computed analogously.
\end{proof}

\subsection{Nondegenerate perturbations of GK structures}

Here we finish the proof of Theorem \ref{thm:nondeg_perturb}, which claims that generalized K\"ahler structures such that $(\J_1, \J_2)$ are both even type can be approximated by nondegenerate structures.  The first step is to show that in a neighbourhood of a generic point $x\in M$, if $\J_1$ is even type, then one can approximate by GK structures such that $\J_1$ is approximated by generalized complex structures of symplectic type.

\begin{prop}[Symplectic perturbation]\label{prop:symp_perturb} Let $(\J_1,\J_2)$ be a generalized K\"ahler structure on $M$ such that $\J_1$ has \textit{even type}. Let $U\subset M$ be the open dense set of points where both $\J_1$ and $\J_2$ have locally constant type. Then for any point $x\in U\subset M$ and $k \geq 2, 0 < \ga < 1$, there exists a closed ball $B(x)$ and an analytic in $t$ family of $C^{k,\alpha}$ generalized K\"ahler structures $(\J_1^t, \J_2^t)$, on $B(x)$ such that 
    \[
    \J_1^0=\J_1,\quad \J_2^0=\J_2,
    \]
    and $\J_1^t$ is of symplectic type for $t\neq 0$. Furthermore, both $\J_1^t$ and $\J_2^t$ are determined by analytic families of $C^{k,\alpha}$ closed spinors.
\end{prop}

\begin{proof}
By the local Darboux theorem for generalized complex structures, Theorem \ref{thm:darboux}, there exists a coordinate neighbourhood $U(x)$ isomorphic to an open subset of $\C^k\times \R^{2(n-k)}$ such that after a $b$-field transform $\J_1$ is given by the standard complex structure on $\C^k$ and by the standard symplectic form on $\R^{2(n-k)}$. We have that $\J_1\Big|_{\C^k}\in \End(\C^k\oplus (\C^k)^*)$ is
\[
\J_1\Big|_{\C^k}=\left(
    \begin{matrix}
    J & 0 \\ 0 & -J^*
    \end{matrix}
\right),
\]
where $J$ is the usual standard complex structure on $\C^k$. Next, since $\J_1$ is of even type, we have that $k$ is even. Let $\beta\in \Lambda^{2,0}(\C^k)$ be a nondegenerate holomorphic Poisson tensor on $\C^k$. Then we can analytically deform $\J_1^0=\J_1$ by varying its restriction on $\C^k$ as follows:
\[
\J_1^t\Big|_{\C^k}=\left(
    \begin{matrix}
    J & 2t\Re(\beta) \\ 0 & -J^*
    \end{matrix}
\right).
\]
If we denote by $\varphi_1=\bar\Theta=d\bar z_1\wedge\dots\wedge d\bar z_k$ a local closed pure spinor determining $\J_1$, then the pure spinor for $\J_1^t$ is given by
\[
\varphi_1^t = e^{-t\i\bar\beta}\lrcorner\,\bar\Theta.
\]
Given that $\beta$ is nondegenerate and of type $(2,0)$ it implies that $\J_1^t$ is of symplectic type as long as $t\neq 0$.

It remains to follow this prescribed variation of $\J_1$ by an appropriate variation of $\J_2$. We can restrict ourselves to a possibly smaller neighbourhood of $x\in M$, where $\J_2$ also admits Darboux coordinates, and is given by a closed pure spinor $\psi_0$:
\[
\J_2=\J_{\psi_0}
\]
Fix a closed ball $B(x)$ in this open neighbourhood. Then by Theorem~\ref{thm:local_stab} we can find an analytic family of closed pure spinors $\psi_t$ on $B(x)$ such that the underlying pair of generalized complex structures
\[
\J_1^t, \J_2^t:=\J_{\psi_t}
\]
is a generalized K\"ahler structure for $t$ sufficiently small.
\end{proof}

Now we use the result of Proposition~\ref{prop:symp_perturb} twice to locally approximate any generalized K\"ahler structure with both $\J_i$ even by sequences of nondegenerate generalized K\"ahler structures.

\begin{thm} \label{thm:nondeg_perturb}
    Let $(\J_1,\J_2)$ be a generalized K\"ahler structure on $M$.  Assume that generalized complex structures $\J_1$ and $\J_2$ of \textit{even type}. Let $U\subset M$ be the open dense set of points where both $\J_1$ and $\J_2$ have locally constant type. Then for any point $x\in U\subset M$ and $k \geq 2, 0 < \ga < 1$, there exists a closed ball $B(x)$, spinors $\psi_1$, $\psi_2$ on $B(x)$ defining $\J_1$ and $\J_2$ and a sequence of \emph{nondegenerate} generalized K\"ahler structures $(\J_1^j, \J_2^j)$ of regularity $C^{k,\ga}$ defined on $B(x)$ such that in $C^{k,\ga}$ topology
    \[
    \lim_{j\to \infty} \J_1^j=\J_1,\quad \lim_{j\to \infty} \J_2^j=\J_2.
    \]
    Furthermore, for each of $\J_1^j$, $\J_2^j$ there exist closed pure spinors $\psi_1^j$, $\psi_2^j$ of regularity $C^{k,\alpha}$ such that in $C^{k,\ga}$
    \[
    \lim_{j\to \infty} \psi_1^j=\psi_1,\quad \lim_{j\to \infty} \psi_2^j=\psi_2.
    \]
\end{thm}
\begin{proof}
By Proposition~\ref{prop:symp_perturb} there exists a closed ball $B(x)$ and an analytic family of generalized K\"ahler structures $(\J_1^t,\J_2^t)$ in $B(x)$, such that $\J_1^t$ is symplectic for $t\neq 0$, $\J_1^0=\J_1$ and both families are given by an analytic family of closed even spinors. Now, pick a sequence $t_i\to 0$, $t_i>0$, and apply Proposition~\ref{prop:symp_perturb} again for each $t_i$ perturbing $\J_2^{t_i}$ into a symplectic generalized complex structure. We obtain an analytic perturbation
\[
\J_1^{t_i,s},\J_2^{t_i,s}, s\in (-\epsilon_i,\epsilon_i)
\]
of $\J_1^{t_i},\J_2^{t_i}$ such that $\J_2^{t_i,s}$ is symplectic for $s\neq 0$. Since being symplectic is an open condition for a generalized complex structure, for $s=s_i>0$ small enough, the other generalized complex structure $\J_1^{t_i,s}$ is still symplectic. Thus we have constructed a sequence $(\J_1^{t_i,s_i},\J_2^{t_i,s_i})$ of nondegenerate generalized K\"ahler structures converging to $(\J^0_1,\J^0_2)$ as claimed. It remains to observe that both perturbations are achieved by variation of the closed pure spinors $\psi_1,\psi_2$ as in Proposition~\ref{prop:symp_perturb}.
\end{proof}

\section{A Calabi conjecture in generalized K\"ahler geometry}
\subsection{The Ricci potential and transgression formula}

Throughout this section $(M^{2n},\JJ_1,\JJ_2)$ is a generalized K\"ahler structure on $(M,H_0)$ with holomorphically trivial canonical bundles. We denote the underlying $d_{H_0}$-closed pure spinors by $\psi_1$ and $\psi_2$. The underlying bihermitian data will be denoted by $(M^{2n}, g, b, I, J)$.

\begin{defn} Let $(M^{2n}, g, b, I, J)$ be a generalized K\"ahler manifold with holomorphically trivial canonical bundles.  The operators $\J_1$ and $\J_2$ coincide on the $+1$-eigenspace of $\mathbb G=-\J_1\J_2$, inducing the same orientation on $TM$ via the isomorphism $C_+\simeq TM$. It follows that the real volume forms $(\psi_1,\bar \psi_1)$ and $(\psi_2,\bar \psi_2)$ are both positively oriented. We will choose the same orientation for the Riemannian volume $dV_g$. Define
\begin{align*}
    \Psi_i = - \log \frac{ (\psi_i, \bar{\psi}_i)}{dV_g}, \qquad \Phi = - \log\frac{ (\psi_1, \bar{\psi}_1)}{(\psi_2, \bar{\psi}_2)} = \Psi_1 - \Psi_2.
\end{align*}
The functions $\Psi_i$ are called \emph{partial Ricci potentials}, and the function $\Phi$ is called the \emph{Ricci potential}.
\end{defn}

\begin{ex}
Let $(M^{2n},g,b,I,J)$ be a nondegenerate generalized K\"ahler manifold. Then up to a $b$-field transform the corresponding closed spinors are $\psi_1= e^{2\Omega + \i F_-}$, $\psi_2=e^{\i F_+}$, and the corresponding partial Ricci potentials are
\[
\Psi_1=\frac{1}{2}\log\det(I-J),\quad \Psi_2=\frac{1}{2}\log\det(I+J).
\]
\end{ex}
The terminology above is justified by the following proposition, which generalizes the classic transgression formula for the Ricci curvature in K\"ahler geometry.

\begin{prop} \label{p:transgression}
Let $(M^{2n}, g, b, I, J)$ be a generalized K\"ahler manifold with holomorphically trivial canonical bundles.  Then we have the following identities:
\begin{equation} \label{f:PsiLeeform}
    (I-J)d\Psi_1= I\theta_I-J\theta_J,\quad (I+J)d\Psi_2= I\theta_I + J\theta_J.
\end{equation}
\begin{equation} \label{f:PhiPoisson}
    \sigma d \Phi = \theta^{\sharp}_I-\theta^{\sharp}_J.
\end{equation}
\begin{equation} \label{f:PhiBismut}
    \rho_I= -\tfrac{1}{2}dJd \Phi,\quad \rho_J= - \tfrac{1}{2}dId \Phi.
\end{equation}
\begin{proof} We prove this via the perturbation method developed in \S \ref{s:perturbations}.  To begin we establish the formulas in the nondegenerate setting.  In particular, for a nondegenerate GK structure, up to conjugation by a closed $b$-field, one has the canonical form $\psi_1 = \psi_- = e^{2 \Omega + \i F_-}$, $\psi_2 = \psi_+ = e^{\i F_+}$.  Using this identification the first two identities follow from \cite[Lemma 3.8]{ASnondeg} and its proof, and the third is \cite[Proposition 3.9]{ASnondeg}.

The strategy now is to apply the perturbation argument of Theorem~\ref{thm:nondeg_perturb} and reduce the proof to the nondegenerate case.  First, let us assume that both $\J_1$ and $\J_2$ are of even type. Let $U\subset M$ be the open dense subset where $\J_1$ and $\J_2$ have constant type, and fix $x\in U$. Then by Theorem~\ref{thm:nondeg_perturb}, in a closed ball $B(x)$ we have find a sequence of nondegenerate generalized K\"ahler structures $(\J_1^j,\J_2^j)$ such that
\[
\lim_{j\to\infty}|\J_1^j-\J_1|_{C^{k,\alpha}}=0,\quad \lim_{j\to\infty}|\J_2^j-\J_2|_{C^{k,\alpha}}=0.
\]
Let $(M,g^j,I^j,J^j)$ be the corresponding bihermitian data. Since the tensors $(g^j,b^j,I^j,J^j)$ depend algebraically on $(\J_1^j,\J_2^j)$, we also have
\begin{align*}
\lim_{j\to\infty} |g^j-g|_{C^{k,\alpha}}=&\ 0,\quad \lim_{j\to\infty} |b^j-b|_{C^{k,\alpha}}=0,\\ 
\lim_{j\to\infty} |I^j-I|_{C^{k,\alpha}}=&\ 0,\quad \lim_{j\to\infty} |J^j-J|_{C^{k,\alpha}}=0.
\end{align*}
Furthermore, as $j\to \infty$ the spinors defining $\J_1^j$ and $\J_2^j$ on $B(x)$ converge in $C^{k,\alpha}$ respectively to $\psi_1$ and $\psi_2$, therefore so do the partial Ricci potentials.  Thus both sides of the identities (\ref{f:PsiLeeform}), (\ref{f:PhiPoisson}) and (\ref{f:PhiBismut}) converge as $j \to \infty$, thus they hold for the original structure in $B(x)$.  Since $x$ is an arbitrary point in an open dense set $U\subset M$, the same identity holds globally on $M$.

If both $\J_1$ and $\J_2$ are of odd type, then we can consider a product of $(M,g,b,I,J)$ with a flat factor $X=(\C, g_{\mathrm{flat}}, I,-I)\times (\C, g_{\mathrm{flat}}, I,I)$, where $I$ is the standard complex structure on $\C$. Thus $X=\C^2$ is equipped with a flat generalized K\"ahler structure of split type, so that $(M\times X,\til{\J_1},\til{\J_2})$ is naturally a generalized K\"ahler manifold with both generalized complex structures of even type.  We can now apply the above argument to conclude that the relevant formulas hold on $M\times X$.  Finally using that $X$ is flat, and both sides of the claimed equations are additive with respect to the Cartesian product, we conclude that the same formula holds on $M$.

If only one of generalized complex structures $\J_1$ and $\J_2$ is odd, say $\J_1$, we perform the same stabilization trick as above, but with the factor $X=\C$ equipped with the standard flat K\"ahler structure, with underlying spinors are $d\bar{z}$ and $e^{\i dx\wedge dy}$.
\end{proof}
\end{prop}

An interesting further geometric consequence of Proposition \ref{p:transgression} is that there are canonically associated flat metrics for the canonical bundles $K_I$ and $K_J$.

\begin{prop} \label{p:trivialcanonical} Let $(M,g,b,I,J)$ be a generalized K\"ahler manifold with holomorphically trivial canonical bundles.  Then the volume form
\begin{equation*}
e^{- \Psi_1 - \Psi_2}dV_g=\frac{(\psi_1,\bar{\psi_1}) (\psi_2,\bar{\psi_2})}{dV_g}
\end{equation*}
induces Chern-flat Hermitian metrics on the (classic) canonical bundles $K_I$ and $K_J$.  %In particular, up to a finite cover of $M$, $K_I$ and $K_J$ are trivial.
\begin{proof} This is an elementary consequence of items (1) and (3) of Proposition \ref{p:transgression} together with the fact that for a pluriclosed structure $(g, I)$ one has $\rho^{C,I} = \rho_I + d I \theta_I$. \end{proof}
\end{prop}

We will also need the following identity relating the scalar curvature to the partial Ricci potentials $\Psi_i$.

\begin{lemma}\label{lm:curvature_lognondeg}
Let $(M^{2n},g,b,I,J)$ be a generalized K\"ahler manifold with holomorphically trivial canonical bundles.  Then
\begin{equation}\label{f:trace_bismut_diff}
	\left(\tr_{\gw_I}\rho_J^B-\tr_{\gw_J}\rho_J^B\right)= \Delta^{C,J}\Psi_2+\frac{1}{6}|H|^2,
\end{equation}
\begin{equation}\label{f:scalar_curvature_spinors}
    R - \frac{1}{12} \brs{H}^2 = -\Delta(\Psi_1+\Psi_2)+\la{d\Psi_1,d\Psi_2},
\end{equation}
where $H=d^c_J\gw_J=-d^c_I\gw_I\in\Lambda^3(M)$ is the intrinsic torsion of $(M,g,b,I,J)$.
\end{lemma}
    \begin{proof}
    Following the perturbation method of Proposition \ref{p:transgression} we are going to prove this identity on a nondegenerate background, i.e., assuming that operators $I\pm J$ are invertible. To simplify the notation we denote the partial Ricci potentials by $\Psi_{+}:=\Psi_2=\frac{1}{2}\log\det(I+J)$ and $\Psi_{-}:=\Psi_1=\frac{1}{2}\log\det(I-J)$.

    We will need the several identities on a generalized K\"ahler manifold $(M, g, I, J)$. First we recall that the Chern and the Riemannian Laplacians are related by an identity
    \begin{equation}\label{f:pf_chern_laplace}
    \Delta^{C,I} f=\Delta f-\la{df,\theta_I}.
    \end{equation}
	Using ~\cite[\S 3]{Ivanovstring}, it follows from the pluriclosed condition for $(g,I)$ and $(g, J)$ that
	\begin{equation}\label{f:pf_norm_H}
	\frac{1}{6}|H|^2=d^*\theta_J+|\theta_J|^2 = d^* \theta_I + \brs{\theta_I}^2
	\end{equation}
	and
	\begin{equation}\label{f:scal_bismut}
	R-\frac{1}{4}|H|^2=2\tr_{\gw_J}\rho_J^B+d^*\theta_J.
	\end{equation}
	Next, according to the proof of \cite[Lemma~5.13]{ASU2}, we have the following general equality  on a GK manifold of symplectic type:
    \[ 2\left({\rm tr}_{F_+} \rho_J - {\rm tr}_{\gw_J} \rho_J \right) = \Delta \Psi_+ - \left\langle d\Psi_+, \theta_I\right\rangle + \frac{1}{2}\left(d^* \theta_J + d^* \theta_I\right) + |\theta_I|^2. \]
    Since $d\Psi_+=(I+J)^{-1}(I\theta_I+J\theta_J)=\frac{1}{2}\left(\theta_I+\theta_J+(I+J)^{-1}(I-J)(\theta_J-\theta_I)\right)$, we have $\la{d\Psi_+,\theta_I-\theta_J}=\frac{1}{2}(|\theta_I|^2-|\theta_J|^2)$. 
    Combining these formulas with~\eqref{f:pf_norm_H} yields
    \begin{align*}
    2\left({\rm tr}_{F_+} \rho_J  - {\rm tr}_{\gw_J} \rho_J \right) &= \Delta \Psi_+ - \left\langle d\Psi_+, \theta_J \right\rangle + \frac{1}{2}\left(d^* \theta_J + d^* \theta_I\right) + \frac{1}{2}\left(|\theta_J|^2 + |\theta_I|^2\right)\\
    &= \Delta^{C, J} \Psi_+ + \frac{1}{6} \brs{H}^2.
    \end{align*}
    Since $2\tr_{F_+}=\tr_{\gw_I}+\tr_{\gw_J}$ this implies equation (\ref{f:trace_bismut_diff}).
	
	From identity~\eqref{f:PsiLeeform}, $2\theta_J=J(I-J)d\Psi_+-J(I-J)d\Psi_-$ and we thus derive
	\begin{equation}\label{f:pf_id6}
	\begin{split}
	    \la {d\Psi_+-\theta_J, d\Psi_--\theta_J} &=
	    \frac{1}{4}\la{J(I-J)(d\Psi_+-d\Psi_-), J(I+J)(d\Psi_+-d\Psi_-)}\\
	   % \la {(I+J)^{-1}I(\theta_I-\theta_J), (I-J)^{-1}I(\theta_I-\theta_J)}\\
	   % &= - \la {I(\theta_I-\theta_J), [I,J]^{-1}I(\theta_I-\theta_J)}\\
	    &=-\frac{1}{4}\la{[I,J]d\Phi,d\Phi}=0,
	\end{split}
	\end{equation}
	where we used that the operator $[I,J]^{-1}$ is skew with respect to $g$.  With~\eqref{f:pf_id6} at hand, we compute,
    \begin{align*}
            R-\frac{1}{12}|H|^2 = &\ 
            2\tr_{\gw_J}\rho_J^B+d^*\theta_J+\frac{1}{6}|H|^2 && \\
            =&\
            2\left(
                \tr_{\gw_I}\rho_J^B-\Delta^{C,J}\Psi_+-\frac{1}{6}|H^2|
            \right)+d^*\theta_J+\frac{1}{6}|H|^2\\
            =&\ 2\tr_{\gw_I}\rho_J^B-2\Delta^{C,J}\Psi_++d^*\theta_J-\frac{1}{6}|H|^2\\
            =&\ 
            2\tr_{\gw_I}\rho_J^B-2\Delta^{C,J}\Psi_+-|\theta_J|^2\\
            =&\ 
            \Delta^{C,I}(\Psi_+-\Psi_-)-2\Delta^{C,J}\Psi_+-|\theta_J|^2\\
            =&\ - \Delta (\Psi_++\Psi_-)+\la{d\Psi_--d\Psi_+,\theta_I}+2\la{d\Psi_+,\theta_J}-|\theta_J|^2
            \\
            =&\ 
            - \Delta (\Psi_++\Psi_-)+\la{d\Psi_++d\Psi_-,\theta_J}+ \la{d\Psi_+-d\Psi_-,\theta_J-\theta_I} -|\theta_J|^2 && \\
            =&\ - \Delta (\Psi_++\Psi_-)+\la{d\Psi_++d\Psi_-,\theta_J} -|\theta_J|^2\\
            =&\ 
            - \Delta (\Psi_++\Psi_-)+\la{d\Psi_+,d\Psi_-},
\end{align*}
as claimed.
\end{proof}

\subsection{The space of generalized K\"ahler structures} \label{s:SOGK}

As discussed in the introduction, the overall space of generalized K\"ahler structures on a given manifold can be quite delicate.  To simplify matters it is natural to fix as background a holomorphic Poisson structure.

\begin{defn} Let $(M^{2n}, I, \sigma)$ be a holomorphic Poisson manifold.  That is, $I$ is an integrable complex structure that $\sigma$ is the real part of an $I$-holomorphic Poisson tensor.  Let
\begin{align*}
 \mathcal{GK}^{I,\sigma} = \left\{ (g,b, I, J)\ \mbox{generalized K\"ahler}\ |\ \tfrac{1}{2} g^{-1} [I, J] = \sigma \right\}.
\end{align*}
Given $m = (g, b, I, J) \in \mathcal {GK}^{I,\sigma}$, let
\begin{align*}
    \mathcal {GK}^{I,\sigma}_m = (\mbox{path component of $\mathcal{GK}^{I,\sigma}$ containing $m$}).
\end{align*}
\end{defn}

Against the background of a holomorphic Poisson manifold, there is a natural extension of the idea of K\"ahler classes in generalized K\"ahler geometry.  The construction naturally unifies the usual notion of K\"ahler class together with the Hamiltonian flow construction of deformations of generalized K\"ahler structures \cite{AGG, GualtieriBranes}.

\begin{defn} \label{d:canonicalfamily} A one-parameter family of generalized K\"ahler structures $\JJ_1^t, \JJ_2^t$ is a \emph{canonical family} if there exists a smooth family $K_t \in \Lambda^2(T^*M)$ such that for all $t$ one has
\begin{align*}
    \dt \JJ_1 = [\JJ_1, K \JJ_1], \qquad \dt \JJ_2 = [\JJ_2, K \JJ_2].
\end{align*}
We further define an equivalence relation where $(\JJ_1, \JJ_2) \sim (\til{\JJ}_1, \til{\JJ}_2)$ if and only if they are connected by a canonical family.  The equivalence class of $(\JJ_1, \JJ_2)$ is denoted by $\mathcal{GK}^{I,\sigma}(m)$ for $m = (g, b, I, J)$.
\end{defn}

\begin{rmk} In \cite{gibson2020deformation} a characterization of canonical families was given in terms of the underlying bihermitian data.  In particular a canonical family will necessary have $I$ and $\sigma$ fixed, and furthermore $dK_t = 0$ and $K_t \in \Lambda^{1,1}_{J_t}(T^*M)$ for all $t$.  In particular this implies that
\begin{align*}
    m \in \mathcal {GK}^{I,\sigma} \implies \mathcal{GK}^{I,\sigma}(m) \subset \mathcal {GK}^{I,\sigma}_m.
\end{align*}
\end{rmk}

\begin{defn} \label{d:GKclass} A one-parameter family of generalized K\"ahler structures $\JJ_1^t, \JJ_2^t$ is an \emph{exact canonical family} if it is a canonical family which further satisfies $K_t = d a_t$.  
Define an equivalence relation where $(\JJ_1, \JJ_2) \sim (\til{\JJ}_1, \til{\JJ}_2)$ if and only if they are connected by an exact canonical family.  The equivalence class of $(\JJ_1, \JJ_2)$ is called the \emph{generalized K\"ahler class of $(\JJ_1, \JJ_2)$}, and is denoted equivalently by $[(\JJ_1, \JJ_2)]$ or $[(g, b, I, J)]$.
\end{defn}

To summarize, given a generalized K\"ahler structure $m = (g, b, I, J) \in \mathcal {GK}^{I,\sigma}$, we have a series of natural inclusions
\begin{align*}
    [m] \subset \mathcal{GK}^{I,\sigma}(m) \subset \mathcal {GK}^{I,\sigma}_m.
\end{align*}
\noindent Very little is known in general concerning the structure of these spaces and inclusions outside of the K\"ahler setting, or the case when $\sigma = 0$.

In the setting of this paper, where the GK structures have holomorphically trivial canonical bundles, we can  also give a specialized definition characterized in terms of twisted deRham classes.  To motivate this definition we first record a lemma showing that canonical deformations preserve the condition of the GK structure having holomorphically trivial canonical bundles, and furthermore preserve the classes of the underlying spinors.

\begin{lemma} \label{l:spinorinvariance} $(M^{2n}, g, b, I, J)$ be a generalized K\"ahler structure with holomorphically trivial canonical bundles, defined by generalized complex structures $\JJ_{\psi_i}$.  Suppose $(\JJ_1^t, \JJ_2^t)$ is a canonical family defined by a one-parameter family of closed two-forms $K_t$, such that $\JJ_i^0 = \JJ_{\psi_i}$.  Furthermore define
\begin{align*}
\psi_i^t = \psi_i + \int_0^t \i K_s \wedge \psi_i ds.
\end{align*}
Then $\JJ^t_i = \JJ_{\psi^t_i}$.
\begin{proof} Let $\psi_i^t$ be defined as above.  We first observe that it follows formally since $K_s$ is closed for all $s$ that $\psi_i^t$ remains $d_{H_0}$-closed.  Thus, for as long as these spinors satisfy the conditions of being pure, i.e., $(\psi^t_i, \psi^t_i) = 0$, and nondegenerate, i.e., $(\psi_i^t, \bar{\psi_i^t}) \neq 0$, they will define generalized complex structures.  We first claim that $\Ker \psi_i^t$ equals the $\i$-eigenspace for $\J_i^t$ for all times, which shows both that the spinors $\psi_i^t$ remain pure, and will satisfy $\JJ_i^t = \JJ_{\psi_i^t}$ as long as they remain nondegenerate.  To this end suppose $e_t$ is a one-parameter family of generalized tangent vectors such that $e_t \in \Ker \psi_i^t$ for all $t$.  Differentiating in $t$ it follows that for any time $t$ one has $\dot{e} + \i K e \in \Ker \psi_i$.  On the other hand this condition guarantees that $\dt (\pi^{0,1}_{\JJ_i} )e = 0$ using the formula for canonical deformations.  Thus we can choose a one-parameter family of bases for $\Ker \{\psi_i^t\}$ at any point and these will in turn be a basis for the $\i$-eigenspace for $\JJ_i^t$, finishing the claim.

It remains to show that the spinors $\psi_i^t$ as defined remain nondegenerate, i.e., $(\psi^t_i, \bar{\psi^t_i}) \neq 0$ for all time.  We first note that since $K_s$ is a smooth family of tensors it follows directly that $(\psi_i^t, \bar{\psi_i^t})$ is bounded above for all $t$.  Since the associated metrics $g_t$ are also controlled, it follows from Proposition \ref{p:trivialcanonical} that there is a bounded function $f_t$ such that $e^{-f_t} (\psi_1^t, \bar{\psi_1^t})$ 
induces a flat metric on the canonical bundle $K_I$.  As this metric is unique up to scale, and $\int_M (\psi_1^t, \bar{\psi_1^t}) = [\psi_1^t] \cdot [\bar{\psi_1^t}] = [\psi_1]\cdot [\bar{\psi_1}]$, it follows that $(\psi_1^t, \bar{\psi_1^t})$ can never vanish.  the argument is the same for $\psi_2^t$.
\end{proof}
\end{lemma}

\begin{defn} \label{d:spinorGKclass} Fix $(M^{2n}, I, \sigma)$ a holomorphic Poisson manifold, a closed $3$-form $H_0$, and let $\alpha$ and $\beta$ denote twisted deRham classes in $H^*_{d_{H_0}}(M, \mathbb C)$.  Define the associated \emph{cohomological generalized K\"ahler class} by
\[ {\mathcal {GK}}^{I, \sigma}_{\alpha, \beta}=\big\{(g, b, I, J) = (\J_{\psi_1}, \J_{\psi_2})\in {\mathcal {GK}}^{I, \sigma},\  \psi_1\in \alpha,\ \psi_2\in \beta\big\}.\]
If $m=(g,b,I, J)\in {\mathcal {GK}}^{I, \sigma}_{\alpha, \beta}$, it follows from Lemma~\ref{l:spinorinvariance} that any exact canonical deformation of $m$ stays in ${\mathcal {GK}}^{I, \sigma}_{\alpha, \beta}$, thus
\[ [m] \subset {\mathcal {GK}}^{I, \sigma}_{\alpha, \beta}. \]
\end{defn}

% In the setting of holomorphically trivial canonical bundles we conjecture that the cohomological generalized K\"ahler class agrees with the equivalence class definition above:

% \begin{conj} \label{c:TCBGKclass} Given $M$ a smooth manifold and $m = (g, b, I, J)$ a generalized K\"ahler structure with Poisson tensor $\gs$ and holomorphically trivial canonical bundles defined by closed spinors $\psi_1, \psi_2$, one has
% \begin{align*}
%     [m] = \mathcal{GK}^{I,\sigma}_{[\psi_1],[\psi_2]}.
% \end{align*}
% \end{conj}

%
%
\subsection{Rigidity and Uniqueness of Generalized Calabi-Yau geometries}

\begin{defn} \label{d:topCY} Let $(M^{2n},g,b,I,J)$ be a generalized K\"ahler manifold with holomorphically trivial canonical bundles.  We say that $(g,b,I,J)$ is a \emph{generalized Calabi-Yau geometry} if 
\begin{align*}
    \Phi := \log \frac{ (\psi_1, \bar{\psi}_1)}{(\psi_2, \bar{\psi}_2)} \equiv \gl
\end{align*}
for some constant $\gl$. 

Notice that if $\psi$ is a $d_{H_0}$-closed form, then its Clifford  involution $s(\psi)$ (see~Definition~\ref{d:mukai}) is $d_{-H_0}$-closed.  It then follows that the constant $\gl$ is determined by the spinor classes $[\psi_i]$:
\begin{align*}
\gl = \log \left[ \frac{\int_M (\psi_1, \bar{\psi}_1)}{\int_M (\psi_2, \bar{\psi}_2)} \right].
\end{align*}
Despite the cohomological interpretation of $\gl$, this constant is not really an invariant of the underlying generalized K\"ahler structure since the spinors $\psi_i$ can be scaled by any nonzero constant without changing the generalized K\"ahler structure.
\end{defn}

\begin{thm} \label{t:rigidity} Suppose $(g, b, I, J)$ is  a generalized Calabi-Yau geometry  on a compact manifold $M$.  Then:
\begin{enumerate}
    \item Both pairs $(g, I)$ and $(g, J)$ are K\"ahler, Ricci-flat.
    \item $(g, b, I, J)$ is the unique generalized Calabi-Yau geometry in  $\mathcal{GK}^{I,\sigma}_{[\psi_1],[\psi_2]}$, up to exact $b$-field transformation.
    \item $(g,b,I,J)$ is the unique generalized Calabi-Yau geometry in $[(g,b,I,J)]$.
\end{enumerate}

\begin{proof} We first prove the K\"ahler rigidity of item (1).  We first note using the second item of Proposition \ref{p:transgression} and the fact that $\Phi$ is constant that $\theta_I = \theta_J$.  Then it follows from the first item of Proposition \ref{p:transgression} and the fact that $d \Psi_1 = d \Psi_2$ that
\begin{align*}
    (I + J) d \Psi_1 = (I + J) \theta_I, \qquad (I - J) d \Psi_1 = (I - J) \theta_I.
\end{align*}
Adding these equations it follows easily that $\theta_I = \theta_J = d \Psi_1$.  Thus using \cite[(3.24)]{Ivanovstring} it follows that
\begin{align*}
    \tfrac{1}{6} \brs{H}^2 = d^*_g \theta_I + \brs{\theta_I}^2 = - \Delta \Psi_1 + \brs{d \Psi_1}^2.
\end{align*}
Since $M$ is closed, it follows from the strong maximum principle that $\Psi_1$ is constant and thus $H \equiv 0$.  Thus $(g, I)$ and $(g, J)$ are K\"ahler, and moreover Ricci-flat by Proposition \ref{p:transgression}.

We now turn to the uniqueness statement (2).   We suppose that $(g, b, I, J)$ defines a generalized K\"ahler structure such that $(g, I)$ and $(g, J)$ are both K\"ahler, Ricci flat, and which is determined by $d_{H_0}$-closed pure spinors $(\psi_1, \psi_2)\in (\alpha,\beta)$.  Using the K\"ahler assumption for $(M, g, I)$, it follows that $H_0 + db= d^c_I \gw_I= 0$. Applying a $b$-field transform with the $2$-form $b$, we can assume without loss of generality that $H_0=0$, $b=0$ and $d\psi_i=0$. We want to show that under these assumptions, $g$ and $J$ are unambiguously determined from the data $(I, \sigma, \alpha, \beta)$, thus establishing the uniqueness of $(g, b, I, J)$ up to a $b$-field transformation. By Proposition \ref{p:KCYspinors}, we can replace $M$ with a finite cover if
necessary and express
\begin{align*}
        \psi_1 =&\ \psi_1^{\gs} \wedge \psi_1^+ \wedge \psi_1^- = e^{2 \Omega + \i F_-} \wedge e^{\i \gw_+} \wedge \bar\Theta_-\\
        \psi_2 =&\ \psi_2^{\sigma} \wedge \psi_2^+ \wedge \psi_2^- = e^{\i F_+} \wedge \bar\Theta_+ \wedge e^{\i \gw_-}.
    \end{align*}
The  above forms $\psi_1$ and $\psi_2$ are $g$-parallel, and thus are the harmonic representatives of $\alpha$ and $\beta$ with respect to $g$. Notice  that by Bochner's theorem the $I$-holomorphic forms $\Theta_-$ and $\Theta_+$  are  respectively the harmonic representatives of the lower degrees of $\alpha$ and $\beta$ with respect to any Ricci-flat K\"ahler metric on $(M, I)$, and thus are determined by $\alpha$ and $\beta$. The triple $(\sigma, \Theta_+, \Theta_-)$ determines the factors $S, X, Y$ in the Beauville-Bogomolov decomposition \eqref{f:splitting-BB} of $(M, g, I)$. Writing \[\Theta_+ =\Theta^X \wedge \Theta^0_+, \qquad \Theta_-=\Theta^Y \wedge \Theta_-^0, \qquad \Omega={\Omega} ^{S} + \Omega^0, \]
\[F_+ =F_+^S +F_+^0, \qquad \gw_+ = \gw^X + \gw_+^0, \qquad \gw_-=\gw^Y + \gw^0, \] with respect to \eqref{f:splitting-BB}, we see that the deRham classes
$2[\Omega_S] + \sqrt{-1}([F_-^S] + [\gw^X])$ and $[F_+^S] + [\gw^Y]$ are determined by $(\alpha, \beta)$. Taking $(1,1)$-parts with respect to $I$ these determine respectively the Aeppli and K\"ahler classes of the Ricci-flat K\"ahler metrics $g_S$ and $g_X, g_Y$. Thus, $(\sigma, \alpha, \beta)$ uniquely determine the  Ricci-flat K\"ahler metrics on $S, X, Y$. We conclude that any other Calabi-Yau generalized K\"ahler structure in $\mathcal{GK}_{\alpha, \beta}^{I, \sigma}$ has a Ricci-flat K\"ahler metric of the form 
$g'=g_S + g_X + g_Y + g_0'$, where $g_0'$ is a flat metric on $\mathbb{T}^{2k}$. 
As $\psi_1$ and $\psi_2$ are parallel forms with respect any such  $g'$, it follows that $(\psi_1, \psi_2)$ are uniquely determined by $(\alpha, \beta)$, and thus so are $(\J_{\psi_1}, \J_{\psi_2})$ and hence $g$ and $J$.

It remains to prove that the $2$-form $b$ of $(g,b,I,J)\in \mathcal{GK}^{I,\gs}_{[\psi_1],[\psi_2]}$ is determined uniquely up to an exact form. As before we may assume that $H_0=0$ so that $b$ above is closed. We aim at proving that it is exact. If $\psi_1,\psi_2$ are the closed spinors corresponding to $(g,0,I,J)$ as above and $\psi_1',\psi_2'$ are the $d$-closed spinors corresponding to $(g,b,I,J)$, then $\psi_i'=e^{-b}\wedge \psi_i$. Since $[\psi_i]=[\psi_i']$ we conclude that $[b]\cup[\psi_i]=0$ in $H^*(M,\C)$.  Let $[b]=[b]_S+[b]_X+[b]_Y+[b]_T$ be the decomposition of $[b]$ according to the Beauville-Bogomolov decomposition. Then from $[b]\cup [\psi_i]=0$ we have
 \[
 -[b]_X + \sqrt{-1} [\omega^X]= \sqrt{-1}[\omega^{X}],
 \]
so $[b]_X=0$.  By a similar reasoning, we have  $[b]_Y=0$ and $[b]_S=0$. Finally, a similar argument yields that on the flat factor $\mathbb{T}^{2k}$, have  
\[
 [b]_T\cup [\Theta_{\pm}^0]=0.
 \] This implies that $[b]_T$ is of type $(2,0)$   and hence is zero as $[b]_T$ is real. We thus conclude   that $[b]=0$ in $H^2(M,\C)$, so that $b$ is exact as claimed.
 
 The proof of item (3) is now straightforward.  Any two generalized Calabi-Yau geometries in $[(g,b,I,J)]$ are by definition connected by a one-parameter family of exact canonical deformations.  Let $(g_s, b_s, I, J_s), s \in [0,1]$ denote such a path, connecting two generalized Calabi-Yau geometries, which are K\"ahler, Ricci-flat by item (1).  It follows from the equations for a canonical deformation \cite[Proposition 3.6]{gibson2020deformation}, and the fact that $g_1 = g_0$ by item (1), that $b_1 - b_0$ is an exact form of type $(2,0) + (0,2)$ with respect to $I$, which necessarily vanishes on the compact K\"ahler background $(M, I)$. 
\end{proof}
\end{thm}

\section{GKRF with trivial canonical bundles}

In this section we establish a number of fundamental properties of the generalized K\"ahler-Ricci flow (GKRF) with holomorphically trivial initial data.  First we recast the flow as an evolution of closed pure spinors, which in particular shows that the cohomological generalized K\"ahler class is preserved.  Thus we have natural Ricci potential functions $\Psi_i$ and $\Phi$ associated to the flow.  Surprisingly, we show in Lemma \ref{l:Riccipotential} that the Ricci potential $\Phi$ solves the associated time-dependent heat flow, while the partial Ricci potentials $\Psi_i$ satisfy the dilaton flow motivated in part by the physical renormalization group flow underlying generalized Ricci flow (cf.\,\cite{Polchinski, Streetsscalar}).  Precursors of this general fact have appeared in \cite{StreetsSTB, StreetsND, ASnondeg}.  This leads to a series of estimates, including a gradient bound for the Ricci potential and a uniform bounds for the appropriate weighted scalar curvatures.  Finally, we define two natural functionals generalizing the Mabuchi energy to this setting, and show that they are monotone along GKRF and bounded, with generalized K\"ahler Calabi-Yau geometries as the only critical points.

\subsection{Background} \label{s:GKRFbackground}

\begin{defn} \label{d:GKRF} A one-parameter family $(g_t, b_t, I, J_t)$ is a solution of \emph{generalized K\"ahler-Ricci flow} (GKRF) if
\begin{gather} \label{f:GKRFpaper}
    \begin{split}
        \dt \gw_I =&\ - 2 \rho_I^{(1,1)}, \qquad \dt b = - 2 \rho_I^{(2,0) + (0,2)} I\qquad \dt J = L_{\theta_J^{\sharp} - \theta_I^{\sharp}} J.
    \end{split}
\end{gather}
This set of equations differs from that stated in (\ref{f:GKRFintro}) by a gauge transformation which fixes the complex structure $I$.  In prior works we referred to this as ``GKRF in the $I$-fixed gauge,'' but we will avoid this terminology here and a solution to GKRF will always mean a solution to (\ref{f:GKRFpaper}).  The equivalence of these formulations is shown in \cite{GKRF} (cf.\,\cite[Ch. 9]{GRFbook}).  As explained in \cite[Theorem 1.4]{gibson2020deformation}, GKRF can be equivalently described in terms of the generalized complex structures as
\begin{align} \label{f:GCflow}
    \dt \JJ_1 = - 2 [\JJ_1, \rho_I \JJ_1], \qquad \dt \JJ_2 = - 2 [\JJ_2, \rho_I \JJ_2].
\end{align}
Furthermore, it follows from Lemma \ref{l:spinorinvariance} that there is a naturally associated flow of pure spinors, namely
\begin{gather} \label{f:spinorflow}
    \begin{split}
        \dt \psi_1 = - 2 \i \rho_I \wedge \psi_1, \qquad \dt \psi_2 = - 2 \i \rho_I \wedge \psi_2.
    \end{split}
\end{gather}
In particular this shows that the property of having holomorphically trivial canonincal bundles is preserved by GKRF.
\end{defn}

\begin{rmk} For a given solution to GKRF, let
\begin{align*}
    \square := \dt - \Delta^{C,I}.
\end{align*}
denote the time-dependent Laplacian associated to the Chern connection.  All evolution equations for this flow are naturally expressed using this Laplacian.  Recall that it is related to the Riemannian Laplacian via
\begin{align*}
  \Delta^{C,I} =&\ \Delta - \theta_I^{\sharp},
\end{align*}
and this extra term precisely corresponds to the gauge relation between (\ref{f:GKRFintro}) and (\ref{f:GKRFpaper}) explained above.
\end{rmk}

\subsection{Ricci potential estimates}

\begin{lemma} \label{l:Riccipotential} Let $(M^{2n}, g, b, I, J)$ be a generalized K\"ahler manifold with holomorphically trivial canonical bundles.  Suppose $(g_t, b_t, I, J_t)$ is the solution to GKRF with this initial data.  Then one has
\begin{gather} \label{f:spinornormev}
\begin{split}
\square \Psi_i =&\ \tfrac{1}{6} \brs{H}^2,\\
\square \Phi =&\ 0.
\end{split}
\end{gather}
\begin{proof} We prove the formula for $\Psi_i$ using the perturbation result of Theorem \ref{thm:nondeg_perturb}.  In particular, it suffices to prove the formula for GKRF in the nondegenerate setting, and then the equation passes to the limit of the nondegenerate approximations.  

In the nondegenerate setting $\Psi_2 = \Psi_+ = \frac{1}{2} \log \det (I+J)$.  Up to a time independent  additive constant, $\Psi_+ = - \log\left(\frac{F_+^{n}}{\gw_I^{n}}\right)$.  We note that it follows from the discussion of the evolution of the underlying spinors above that $\frac{\partial}{\partial t} F_+ = - 2\rho_I$  (cf.\,also \cite[Lemma 5.2]{ASnondeg}).  Using this and $\frac{\partial}{\partial t} \gw_I = -2 \rho_I^{1,1}$ we have
\begin{equation}\label{e:evolution-Psi}
\frac{\partial}{\partial t} \Psi_+ = 2 \left(\tr_{F_+} \rho_I - \tr_{\gw_I} \rho_I \right)=\tr_{\gw_J}\rho_I-\tr_{\gw_I}\rho_I.
\end{equation}
According to~\eqref{f:trace_bismut_diff} (with the roles of $I$ and $J$ swapped) this implies
\[
\frac{\del}{\del t}\Psi_+=\Delta^{C,I}\Psi_++\frac{1}{6}|H|^2,
\]
as stated.  The case of $\Psi_1 = \Psi_-$ is analogous. This finishes the proof of the evolution of $\Psi_i$, and the formula for $\Phi$ is a formal consequence.
\end{proof}
\end{lemma}

\begin{lemma} \label{l:gradientPhi} Let $(M^{2n}, g, b, I, J)$ be a generalized K\"ahler manifold with holomorphically trivial canonical bundles.  Suppose $(g_t, b_t, I, J_t)$ is the solution to GKRF with this initial data.  Then
\begin{gather} \label{f:gradPhiev}
\begin{split}
\square \brs{\N \Phi}^2 =&\ -2 \brs{\N^2 \Phi}^2 - \tfrac{1}{2} \IP{H^2, \N \Phi \otimes \N \Phi}.
\end{split}
\end{gather}
\begin{proof} The result is a formal consequence of Lemma \ref{l:Riccipotential} and the Bochner formula (cf.\,\cite[Lemma 4.3]{StreetsND}).
\end{proof}
\end{lemma}

\begin{prop} \label{p:gradientdecay} Let $(M^{2n}, g, b, I, J)$ be a generalized K\"ahler manifold with holomorphically trivial canonical bundles.  Suppose $(g_t, b_t, I, J_t)$ is the solution to GKRF with this initial data.  Then
\begin{gather} \label{f:gradientestimate}
\sup_{M \times \{t\}} \left( \Phi^2 + t \brs{\N \Phi}^2 \right) \leq \sup_{M \times \{0\}} \Phi^2.
\end{gather}
\begin{proof} Let $F(x,t) = t \brs{\N \Phi}^2 + \Phi^2$.  By Lemmas \ref{l:Riccipotential} and \ref{l:gradientPhi} it follows that $\square F \leq 0$, and the result follows by the maximum principle.
\end{proof}
\end{prop}

\subsection{Scalar curvature estimates} \label{ss:scalar}

In \cite{Streetsscalar} the third named author proved a family of scalar curvature monotonicities along the generalized Ricci flow, with the key new ingredient being an auxiliary solution to a certain PDE, the \emph{dilaton flow}:
\begin{align} \label{f:dilatonflow}
\square \phi = \tfrac{1}{6} \brs{H}^2.
\end{align}
In \cite[Proposition 2.4]{Streetsscalar} it was shown that if $(g_t, H_t)$ is a solution to generalized Ricci flow and $\phi_t$ is a solution of the generalized Ricci flow, then the weighted scalar curvature
\begin{align*}
    R^{H,\phi} = R - \tfrac{1}{12} \brs{H}^2 + 2 \gD \phi - \brs{\N \phi}^2
\end{align*}
satisfies
\begin{align*}
    \square R^{H,\phi} =&\ 2 \brs{\Rc^{H,\phi}}^2,
\end{align*}
 where
 \begin{align*}
     \Rc^{H,\phi} =&\ \Rc^g - \tfrac{1}{4} H^2 + \N^2 \phi - \tfrac{1}{2} ( d^*_g H + i_{\N f} H).
 \end{align*}
 
\begin{prop} \label{p:scalarmon} Let $(M^{2n}, g, b, I, J)$ be a generalized K\"ahler manifold with holomorphically trivial canonical bundles.  Suppose $(g_t, b_t, I, J_t)$ is the solution to GKRF with this initial data.  Then
\begin{align*}
    \square R^{H,\Psi_{i}} = 2 \brs{\Rc^{H,\Psi_{i}}}^2.
\end{align*}
Moreover,
\begin{align*}
    \inf_{M \times \{t\}} R^{H,\Psi_i} \geq \inf_{M \times \{0\}} R^{H,\Psi_i}.
\end{align*}
\begin{proof} The evolution equation is a formal consequence of Lemma \ref{l:Riccipotential} and \cite[Proposition 2.4]{Streetsscalar}, and the estimate follows from the maximum principle.
\end{proof}
\end{prop}

\begin{prop} \label{p:scalarupper} Let $(M^{2n}, g, b, I, J)$ be a generalized K\"ahler manifold with holomorphically trivial canonical bundles.  Suppose $(g_t, b_t, I, J_t)$ is the solution to GKRF with this initial data.  Then
\begin{align*}
    \sup_{M \times \{t\}} R^{H,\Psi_1} \leq - \inf_{M \times \{0\}} R^{H,\Psi_2}, \qquad \sup_{M \times \{t\}} R^{H,\Psi_2} \leq - \inf_{M \times \{0\}} R^{H,\Psi_1}.
\end{align*}
\begin{proof} By Lemma \ref{lm:curvature_lognondeg} it follows that
\begin{align*}
    R^{H,\Psi_1} = - \gD \Phi + \IP{d \Psi_1, d \Phi}, \qquad R^{H,\Psi_2} = \gD \Phi - \IP{d \Psi_2, d \Phi}.
\end{align*}
Thus
\begin{align*}
    R^{H,\Psi_1} + R^{H,\Psi_2} = - \brs{d \Phi}^2 \leq 0.
\end{align*}
Using the lower bound on $R^{H,\psi_i}$ from Proposition \ref{p:scalarmon}, the result follows.
\end{proof}
\end{prop}

\subsection{Mabuchi energy monotonicity}

Here we give an extension of the Mabuchi energy to the setting of generalized K\"ahler geometry with holomorphically trivial canonical bundles.  We show that it is monotone along GKRF, with the only possible critical points generalized Calabi-Yau geometries which by Theorem \ref{t:rigidity} are actually K\"ahler, Calabi-Yau.

\begin{defn} Let $(M^{2n}, g, b, I, J)$ be a generalized K\"ahler manifold with holomorphically trivial canonical bundles.  Define the associated \emph{Mabuchi energies} by
\begin{align*}
    \mathcal M_i(g, b, I, J) = \int_M \Phi e^{-\Psi_i} dV_g = \int_M \Phi (\psi_i, \bar{\psi}_i).
\end{align*}
\end{defn}

An elementary application of Jensen's inequality shows that these energies admit one-sided bounds:

\begin{lemma} \label{l:Mabuchibounds} Let $(M^{2n}, g, b, I, J)$ be a generalized K\"ahler manifold with holomorphically trivial canonical bundles.  Then
$$\mathcal M_1\leq\gl\int_M e^{-\Psi_1}dV_g, \qquad \mathcal M_2\geq \gl\int_M e^{-\Psi_2}dV_g.$$
\end{lemma}
\begin{proof} Using the definitions of $\gl$ and $\Psi_i$ and Jensen's inequality we obtain
\begin{align*}
e^{\gl} = \frac{\int e^{-\Psi_2}dV_g}{\int e^{-\Psi_1}dV_g} = \frac{\int e^{\Phi} e^{-\Psi_1}dV_g}{\int e^{-\Psi_1}dV_g} \geq \exp \left( \frac{\int \Phi e^{-\Psi_1}dV_g}{\int e^{-\Psi_1}dV_g} \right).
\end{align*}
This implies that $\mathcal M_1\leq \gl\int e^{-\Psi_1}dV_g.$ Similarly, one has $\mathcal M_2\geq\gl\int e^{-\Psi_2}dV_g$.
\end{proof}

The bounds of the Mabuchi energies above are in fact preserved along GKRF:

\begin{lemma}\label{l:ev-volume} Let $(M^{2n}, g, b, I, J)$ be a generalized K\"ahler manifold with holomorphically trivial canonical bundles.  Suppose $(g_t, b_t, I, J_t)$ is the solution to GKRF with this initial data.  Then
\begin{align*}\frac{d}{dt}\Big(\int e^{-\Psi_1}dV_g\Big)=0, \qquad 
\frac{d}{dt}\Big(\int e^{-\Psi_2}dV_g\Big)=0.\end{align*}
\end{lemma}
\begin{proof}
    By the evolution equation (\ref{f:spinorflow}) for the spinors and the transgression formula (\ref{f:PhiBismut}), it follows that the cohomology classes of $\psi_i$ are preserved. The the lemma is direct consequence of the definition
    $e^{-\Psi_1}dV_g=(\psi_1,\bar\psi_1)$ and Stokes Theorem.
\end{proof}

To show the monotonicity of the Mabuchi energies along GKRF, we require several key identities for the partial Ricci potentials which are used in the calculation, which are proved via the perturbation method.
\begin{lemma}\label{psi-theta} Let $(M^{2n}, g, b, I, J)$ be a generalized K\"ahler manifold with holomorphically trivial canonical bundles.  Then
\begin{align*}
&\langle d\Phi,d\Psi_2-\theta_I\rangle=\frac{1}{2}\Big(\langle  Id\Phi,Jd\Phi\rangle-|d\Phi|^2\Big),\\
&\langle Jd\Phi, Id\Psi_2-I\theta_I\rangle=\frac{1}{2}\Big( |d\Phi|^2-\langle Id\Phi,Jd\Phi\rangle\Big),\\
&\langle d\Phi,d\Psi_1-\theta_I\rangle=\frac{1}{2}\Big(\langle  Id\Phi,Jd\Phi\rangle+|d\Phi|^2\Big),\\
&\langle Jd\Phi, Id\Psi_1-I\theta_I\rangle=\frac{1}{2}\Big(\langle  Id\Phi,Jd\Phi\rangle+|d\Phi|^2\Big).
\end{align*}
\end{lemma}

\begin{proof}
It follows from~\eqref{f:PsiLeeform} that 
\[
\theta_I=-\frac{1}{2}I\left(
    (I-J)d\Psi_1+(I+J)d\Psi_2
\right).
\]
It is now straightforward to verify all 4 identities, keeping in mind that $\Phi=\Psi_1-\Psi_2$.
\end{proof}

\begin{prop}\label{p:Mabuchimon} Let $(M^{2n}, g, b, I, J)$ be a generalized K\"ahler manifold with holomorphically trivial canonical bundles.  Suppose $(g_t, b_t, I, J_t)$ is the solution to GKRF with this initial data.  Then
\begin{align*}
\frac{d}{dt} \mathcal M_1 =\int |d\Phi|^2_ge^{-\Psi_1}dV_g, \qquad 
\frac{d}{dt} \mathcal M_2 =-\int |d\Phi|^2_ge^{-\Psi_2}dV_g.\end{align*}
\end{prop}

\begin{proof} We first claim that $\Psi_2$ satisfies
\begin{equation}\label{hhhhh}
    \frac{\partial\Psi_2}{\partial t}=-2\tr_{\omega_I}\rho_I-\frac{1}{2}\Delta^{C,J}\Phi-\frac{1}{2}\tr_{\omega_I}dd^c_J\Phi.
\end{equation}
Aiming to use the perturbation method, we first note that in the nondegenerate case,
by identity (\ref{e:evolution-Psi}) in Lemma \ref{l:Riccipotential}, one has 
$$\frac{\partial}{\partial t} \Psi_+ = 2 \left(\tr_{F_+} \rho_I - \tr_{\omega_I} \rho_I \right).$$ 
It follows from the identity $\rho_I=-\frac{1}{2}dJd\Phi$ (cf.\,Proposition \ref{f:PsiLeeform}) that
$$\frac{\partial\Psi_+}{\partial t}=-2\tr_{\omega_I}\rho_I-\frac{1}{2}\tr_{\omega_I+\omega_J}dd^c_J\Phi=-2\tr_{\omega_I}\rho_I-\frac{1}{2}\Delta^{C,J}\Phi-\frac{1}{2}\tr_{\omega_I}dd^c_J\Phi.$$Applying the perturbation argument equation (\ref{hhhhh}) follows.

We also calculate the evolution of $\mu_2:=\Phi e^{-\Psi_2}dV_g$, using the flow equations for $g$ and Lemma \ref{l:Riccipotential},
\begin{equation*}
\dt\mu_2=\Delta^{C,I}\Phi e^{-\Psi_2}dV_g+\Phi(-\frac{\partial\Psi_2}{\partial t}-2\tr_{\omega_I}\rho_I) e^{-\Psi_2}dV_g.
\end{equation*}
It follows from identity (\ref{hhhhh}) that
\begin{align*}
    \frac{d}{dt} {\mathcal M_2}=\int\Delta^{C,I}\Phi e^{-\Psi_2}+\Phi\Big(\frac{1}{2}\Delta^{C,J}\Phi+\frac{1}{2}\tr_{\omega_I}dd^c_J\Phi\Big)e^{-\Psi_2}dV_g.
\end{align*}
We further compute,
\begin{equation}\label{d}
    \begin{split}
        %\int \Delta^{C,J}\Phi e^{-\Psi_2}dV_g=&\ 
        \int \Delta^{C,I}\Phi e^{-\Psi_2}dV_g =&\ \int\Big( \Delta\Phi-\langle \theta_I,d\Phi\rangle\Big) e^{-\Psi_2}dV_g\\
        =&\ \int \langle d\Phi, d\Psi_2-\theta_I \rangle e^{-\Psi_2}dV_g\\
        =&\frac{1}{2}\int\Big(\langle  Id\Phi,Jd\Phi\rangle-|d\Phi|^2\Big)e^{-\Psi_2}dV_g\\
        =&\ \frac{1}{2}\int\langle Id\Phi, Jd\Phi\rangle e^{-\Psi_2}dV_g\\
        &\ \qquad +\frac{1}{2}\int \Phi\Big(\Delta^{C,I}\Phi+\langle d\Phi,\theta_I\rangle-\langle d\Phi,d\Psi_2\rangle\Big) e^{-\Psi_2}dV_g\\
        =&\ \frac{1}{2}\int\langle Id\Phi, Jd\Phi\rangle e^{-\Psi_2}dV_g\\
        &\ \qquad \frac{1}{2}\int \Phi\Big(\Delta^{C,I}\Phi +\frac{1}{2} |d\Phi|^2-\frac{1}{2}\langle Id\Phi, Jd\Phi\rangle\Big)e^{-\Psi_2}dV_g,
    \end{split}
\end{equation}
where  the third equation follows from the first identity from Lemma \ref{psi-theta}, the fourth equation follows from integration by part and the fifth equation  follows from the first identity of Lemma \ref{psi-theta}.   We also compute the further identity
\begin{equation}\label{e}
    \begin{split}
        \int\langle Id\Phi, & Jd\Phi\rangle e^{-\Psi_2}dV_g +\tr_{\omega_I}(dd^c_J\Phi)\Phi e^{-\Psi_2}dV_g\\
        &=\int\Big(\langle d\Phi\wedge Jd\Phi,\omega_I\rangle+\Phi\langle dJd\Phi,\omega_I\rangle\Big)e^{-\Psi_2}dV_g\\
        &=\int \langle d(\Phi Jd\Phi),\omega_I\rangle e^{-\Psi_2}dV_g\\
        &=\int \langle \Phi Jd\Phi, d^*(e^{-\Psi_2}\omega_I)\rangle dV_g\\
        &=\int \Phi\Big(\langle Jd\Phi, d^*\omega_I\rangle+\langle d\Psi_2\wedge Jd\Phi,\omega_I\rangle\Big)e^{-\Psi_2}dV_g\\
      &=\int \Phi \langle Jd\Phi, Id\Psi_2-I\theta_I\rangle e^{-\Psi_2}dV_g\\
      &=\frac{1}{2}\int \Phi\Big(|d\Phi|^2-\langle Id\Phi, Jd\Phi\rangle\Big)e^{-\Psi_2}dV_g.
    \end{split}
\end{equation}
where the last line follows from the second identity of Lemma \ref{psi-theta}.  Using the identity $\Delta^{C,J}\Phi=\Delta^{C,I}\Phi$ and combining identities (\ref{d}), (\ref{e}) yields
\begin{equation*}
    \frac{d}{dt} \mathcal M_2 =\int \Phi  \, \Big[\Delta_{g}^{C, I} \Phi  + \frac{1}{2}\left(|d\Phi|^2 - \langle Jd\Phi, I d\Phi \rangle\right)\Big]e^{-\Psi_2}dV_g.
\end{equation*}
 Finally we compute
\begin{equation*}
\begin{split}
\int\Phi \Delta_{g}^{C, I} \Phi e^{-\Psi_2}dV_g &=
\int\Phi \Delta_{g} \Phi e^{-\Psi_2}dV_g-\int\langle \theta_I,d\Phi\rangle e^{-\Psi_2}\Phi dV_g\\
&=-\int|d\Phi|^2e^{-\Psi_2}dV_g+\int\langle d\Phi, d\Psi_2-\theta_I\rangle e^{-\Psi_2}\Phi dV_g\\
&=-\int|d\Phi|^2e^{-\Psi_2}dV_g+\int\frac{1}{2}\Big(\langle  Id\Phi,Jd\Phi\rangle-|d\Phi|^2\Big)e^{-\Psi_2}dV_g,
\end{split}
\end{equation*}
where the last line uses the first identity of Lemma \ref{psi-theta}. Rearranging finishes the claim for $\mathcal M_2$.  The calculation for $\mathcal M_1$ is similar and left to the reader.
\end{proof}

\subsection{Main result}
At his point we have all the ingredients to prove our first main result on GKRF on a generalized K\"ahler manifold with holomorphically trivial canonical bundles.

\begin{thm} \label{t:GKRFprops} Let $(M^{2n}, g, b, I, J)$ be a compact generalized K\"ahler manifold with holomorphically trivial canonical bundles.  Let $(g_t, b_t, I, J_t)$ be the solution to generalized K\"ahler-Ricci flow with this initial data.  The following hold:
\begin{enumerate}
    \item The canonical bundles of $\JJ_i^t$ are holomorphically trivial for all time, defined by closed pure spinors $\psi_1^t \in [\psi_1^0] = \ga, \psi_2^t \in [\psi_2^0] = \gb$.  Furthermore
    \begin{align*}
    (g_t, b_t, I, J_t) \in [(g_0,b_0,I,J_0)] \subset \mathcal{GK}^{I,\sigma}_{\ga,\gb}.
    \end{align*}
    \item For any smooth existence time $t$ one has Ricci potential bounds
    \begin{align*}
        \sup_{M \times \{t\}} \left( \Phi^2 + t \brs{\N \Phi}^2 \right) \leq \sup_{M \times \{0\}} \Phi^2.
    \end{align*}
    \item For any smooth existence time $t$ one has scalar curvature bounds
    \begin{align*}
        \inf_{M \times \{0\}} R^{H,\Psi_1} \leq&\ R^{H,\Psi_1} (\cdot, t) \leq - \inf_{M \times \{0\}} R^{H,\Psi_2},\\
        \inf_{M \times \{0\}} R^{H,\Psi_2} \leq&\ R^{H,\Psi_2} (\cdot, t) \leq - \inf_{M \times \{0\}} R^{H,\Psi_1}.
    \end{align*}
    \item There exist Mabuchi-type functionals $\mathcal M_i := \int_M \Phi (\psi_i, \bar{\psi_i})$ whose only critical points are generalized Calabi-Yau geometries, and which are bounded and monotone along GKRF.
\end{enumerate}
\end{thm}

\begin{proof}%[Proof of Theorem \ref{t:GKRFprops}]
As discussed at the start of this section the flow will be defined by global closed pure spinors as long as it is smooth.  By the evolution equation (\ref{f:spinorflow}) for the spinors and the transgression formula (\ref{f:PhiBismut}) it follows that the cohomology classes of $\psi_i$ are preserved.  Furthermore as explained in equation (\ref{f:GCflow}), the GKRF defines a canonical family as in Definition \ref{d:canonicalfamily}, which in fact by (\ref{f:PhiBismut}) is exact, thus by definition the generalized K\"ahler class is preserved.  In particular $I$ and $\gs$ are preserved, so the flow remains in $\mathcal{GK}^{I,\sigma}_{[\psi_1^0],[\psi_2^0]}$, as claimed, finishing the proof of item (1).  Item (2) is Proposition \ref{p:gradientdecay}, item (3) follows from Propositions \ref{p:scalarmon} and \ref{p:scalarupper}.  The monotonicity of the Mabuchi energies follows from Proposition \ref{p:Mabuchimon}, whereas the one-sided bounds follow from Lemma \ref{l:Mabuchibounds}.  These bounds are in terms of the spinor masses which are canonically associated to the spinor classes, which by Lemma \ref{l:ev-volume} are invariant under the GKRF.
\end{proof}

\section{Global existence on K\"ahler backgrounds}

The main goal of this section is to prove item $(1)$ of Theorem \ref{t:longcon}.  First, let we fix some notations. Let $(M^{2n}, g, b, I, J)$ be a compact generalized K\"ahler manifold with holomorphically trivial canonical bundles.  We further assume that $(M,I)$ is K\"ahler.  We shall prove the long time existence of GKRF under this setup.  We remark that if the initial data is nondegenerate, the long time existence is proved in \cite{ASnondeg}.  The proof here is similar to the proof of the nondegenerate case, where the key new input is the transgression formula for the Ricci curvature (Proposition \ref{p:transgression}), which yields a number of useful a priori estimates.  The reason for the restriction to K\"ahler backgrounds is that this yields a number of sharp one-sided differential inequalities for some delicate quantities in the $1$-form reduced flow (cf.\,Lemma \ref{l:reducedestimates}), which in general have reaction terms which are difficult to control.

\subsection{The \texorpdfstring{$1$}{1}-form/scalar reduction}\label{s:reducedflow}

In this subsection we record various evolution equations and estimates for a reduction of the pluriclosed flow to a certain system coupling a
$(1,0)$-form with a scalar.  We exploit the
Calabi-Yau background to simplify the background terms needed to define this
reduced equation.  To describe this reduction, first note that since $(M, I)$ is assumed K\"ahler it follows from the theorem of Demailly-Paun \cite{DemaillyPaun} (cf.\,\cite[Proposition 6.1]{ASnondeg}) that there exists a K\"ahler metric in the Aeppli class of $\gw_I$.  By Proposition \ref{p:trivialcanonical} we know that $c_1(M, I) = 0$, thus by Yau's theorem \cite{YauCC} there exists a Calabi-Yau metric $\gw_{CY}$ in the Aeppli class of $\gw_I$.  In particular, there exists a Calabi-Yau metric $\gw_{CY}$ and $\ga \in
\Lambda^{1,0}_I(T^*M)_{\C}$ such that
\begin{align*}
\gw_I = \gw_{CY} + \del \bga + \delb \ga.
\end{align*}

\begin{lemma} \label{oneformreduction} Let $(M, I, g_t)$ be a
solution to pluriclosed flow, and suppose $\ga_t \in \Lambda^{1,0}_I(T^*M_{\C})$
satisfies
\begin{gather} \label{alphaflow}
\begin{split}
\dt \ga =&\ \delb^*_{\gw_t} \gw_t - \frac{{\i}}{2}\del \log \tfrac{\det
g_t}{\det
g_{CY}}\\
\ga(0) =&\ \ga_0,
\end{split}
\end{gather}
then the one-parameter family of pluriclosed metrics $\gw_{\ga} =
\gw_{CY}
+
\delb
\ga + \del \bga$ is the given solution to pluriclosed flow.
\begin{proof} This is a simple modification of (\cite[Lemma 3.2]{StreetsPCFBI}).
\end{proof}
\end{lemma}

We note that the natural local decomposition of a pluriclosed metric as $\gw =
\gw_{CY}+
\del \bga +
\delb \ga$ is not canonical, as one may observe that $\ga + \del f$ describes
the same K\"ahler form for $f \in C^{\infty}(M, \mathbb R)$.  Due to this invariance,  the equation (\ref{alphaflow}) is not parabolic, and there are infinite dimensional families of equivalent solutions.  In \cite{StreetsPCFBI} the third author
resolved this
ambiguity by giving a further reduced description of (\ref{alphaflow}) which is
parabolic.  In particular, as exhibited in \cite[Proposition 3.9]{StreetsPCFBI} in
the case when the background metric is fixed and K\"ahler, if one has a family
of
functions $f_t$ and $(1,0)$-forms $\gb_t$ which satisfy
\begin{gather} \label{decflow}
\begin{split}
\square \gb =&\ - T \circ \delb \gb,\\
\square f =&\ \tr_{g_t} g_{CY} + \log \tfrac{\det
g_t}{\det g_{CY}},\\
\ga_0 =&\ \gb_0 - \frac{\i}{2} \del f_0,
\end{split}
\end{gather}
then $\ga_t := \gb_t - \frac{\i}{2} \del f_t$ is a solution to (\ref{alphaflow}).  The
term $T \circ \delb \gb$ is defined by
\begin{align*}
(T \circ \delb \gb)_i = g^{\bl k} g^{\bq p} T_{ik\bq} \N_{\bl} \gb_p.
\end{align*}

\subsection{Global existence}

In this subsection we establish global
existence of the generalized K\"ahler-Ricci flow on a Calabi-Yau background following the argument of \cite{ASnondeg, StreetsND}.  We first record a lemma with useful differential inequalities satisfied by the reduced flow (\ref{decflow}), and a lemma with general evolution equations for pluriclosed flow.

\begin{lemma} \label{l:reducedestimates} (cf.\,\cite[\S 6]{ASnondeg}, \cite{StreetsPCFBI}) Given a
solution to (\ref{decflow})
as above, one has
\begin{gather*}
\begin{split}
\square \tfrac{\del f}{\del t} =&\ \IP{\tfrac{\del
g}{\del t}, \delb \gb + \del \bgb},\\
\square \brs{\gb}^2 =&\ - \brs{\N \gb}^2 -
\brs{\bar{\N} \gb}^2 -
\IP{Q, \gb \otimes \bar{\gb}} + 2 \Re \IP{\gb, T \circ \delb \gb} \leq 0,\\
\square \brs{\del \ga}^2 =&\ - \brs{\N \del \ga}^2 -
\brs{T}^2 - 2
\IP{Q, \del \ga \otimes \delb \bga} \leq 0.
\end{split}
\end{gather*}
where
\begin{align*}
Q_{i \bj} = g^{\bl k} g^{\bq p} T_{i k \bq} T_{\bj \bl p}.
\end{align*}
\end{lemma}

\begin{lemma} \label{volumeformev} (cf.\,\cite[\S 6]{ASnondeg}, \cite{StreetsPCFBI}) Let $(M^{2n}, I, g_t)$ be a solution to
pluriclosed flow, and suppose $h$ is another Hermitian metric on $(M,I)$.  Then
\begin{align*}
\square \log \tfrac{\det g_t}{\det h} =&\ \brs{T}^2 -
\tr_g
\rho_C(h),\\
\square \log \tr_h g \leq&\ \brs{T}^2 + C \tr_g h.
\end{align*}
\end{lemma}

\begin{prop} \label{lteprop} Let $(M^{2n}, g, b, I, J)$ be a compact generalized K\"ahler manifold with holomorphically trivial canonical bundles.   Suppose further $(M,I)$ is K\"ahler,
then 
\begin{enumerate}\item The generalized K\"ahler Ricci flow with
initial condition $(g,b,I,J)$ exists on $[0,\infty)$,
\item One has $\lim_{t \to \infty} \mathcal M_i(t) = \gl \int e^{-\Psi_i}dV_g = \gl \int_M (\psi_i, \bar{\psi}_i)$.
\end{enumerate}

\begin{proof} It follows from Proposition \ref{p:trivialcanonical} that $c_1(M, I) = 0$.   Since $(M, I)$ is also K\"ahler, we may construct a solution $(\gb_t,f_t)$ to (\ref{decflow}) as above.  The proof is similar to that of \cite[Proposition 6.10]{ASnondeg}, so we only sketch certain arguments.

From the general theory of pluriclosed flow (\cite{StreetsPCFBI, JordanStreets, JFS}), it suffices to establish uniform parabolicity of the flow and an upper bound for $\brs{\del \ga}^2$.  This latter upper bound follows directly from the maximum principle applied to the evolution equation from Lemma \ref{l:reducedestimates}.  The first step is to obtain a uniform estimate for the Riemannian volume form.  We apply
Lemma \ref{volumeformev} choosing $h = g_{CY}$, so that $\rho_C(h)
=
\rho_C(g_{CY}) = 0$, to obtain
\begin{align*}
\square \log \tfrac{\det g_t}{\det g_{CY}}
=
\brs{T}^2
\geq 0.
\end{align*}
The maximum principle then implies
\begin{align*}
\inf_{M \times \{t\}} \log \tfrac{\det g_t}{\det g_{CY}} \geq&\
\inf_{M \times
\{0\}} \log \tfrac{\det g_t}{\det g_{CY}}.
\end{align*}
Now define
\begin{align*}
W_1 = \log \tfrac{\det g_t}{\det g_{CY}} + \brs{\del \ga}^2.
\end{align*}
Combining Lemma \ref{l:reducedestimates} with Lemma \ref{volumeformev}
we obtain
\begin{align*}
\square W_1 \leq&\ 0.
\end{align*}
By the maximum principle we obtain
\begin{align*}
\sup_{M \times \{t\}} \log \tfrac{\det g_t}{\det g_{CY}} \leq
\sup_{M \times
\{t\}} W_1 \leq \sup_{M \times \{0\}} W_1 \leq C.
\end{align*}
This establishes uniform upper and lower bounds on the volume form.  

Next we establish a uniform metric upper bound in terms of a bound for the scalar potential $f$.  To that end fix some $A > 0$ and define
\begin{align*}
W_2 = \log \tr_{g_{CY}} g + \brs{\del \ga}^2 - A f.
\end{align*}
Using Lemma
\ref{l:reducedestimates}, \ref{volumeformev}, and equation (\ref{decflow}), we have
\begin{align*}
\square W_2 \leq&\ \left( C - A \right) \tr_{g}
g_{CY} -
A \log \tfrac{\det g_t}{\det g_{CY}}.
\end{align*}
Choosing $A=C$ and using the
uniform bound for the volume form then yields
\begin{align*}
\square W_2 \leq C.
\end{align*}
Applying the maximum principle one has
\begin{align*}
\sup_{M \times \{t\}} \log \tr_{g_{CY}} g - A f \leq&\ \sup_{M
\times \{t\}} W_2
\leq \sup_{M \times \{0\}} W_2 + C t \leq C(1 + t).
\end{align*}
Rearranging this we obtain
\begin{align*}
\sup_{M \times \{t\}} \tr_{g_{CY}} g \leq&\ \sup_{M \times \{t\}}
e^{C (1 + t +
f)}.
\end{align*}

To finish the proof it suffices to obtain an upper bound for $f$.  Since we are only
concerned with finite time intervals, it suffices to estimate $\tfrac{\del
f}{\del t}$.  To that end fix constants $A_1, A_2$ and let
\begin{align*}
 W_3 =&\ \tfrac{\del f}{\del t} + \brs{\gb}^2 - A_1
\log \tfrac{\det g}{\det g_{CY}} + A_2 \brs{\N \Phi}^2.
\end{align*}
Combining Lemmas \ref{l:reducedestimates} and \ref{volumeformev}
with Proposition \ref{l:gradientPhi} we obtain
\begin{gather*}
\begin{split}
 \square W_3 =&\ \IP{\tfrac{\del g}{\del t}, \delb \gb
+
\del \bgb} + \left[ - \brs{\N \gb}^2 -
\brs{\bar{\N} \gb}^2 -
\IP{Q, \gb \otimes \bar{\gb}} + 2 \Re \IP{\gb, T \circ \delb \gb} \right]\\
&\ - A_1 \brs{T}^2 + A_2 \left[ -2 \brs{\N^2 \Phi}^2 - \tfrac{1}{2} \IP{\HH, \N
\Phi \otimes \N \Phi}\right].
\end{split}
\end{gather*}
To estimate the right hand side, we first note that by applying the maximum principle to the evolution equation for $\brs{\gb}^2$ we obtain a uniform bound for this quantity.  Thus, by the Cauchy-Schwarz inequality we obtain
\begin{align*}
 2 \Re \IP{\gb, T \circ \delb \gb} \leq&\ C \brs{T} \brs{\bar{\N} \gb} \leq
\tfrac{1}{2} \brs{\bar{\N} \gb}^2 + C \brs{T}^2.
\end{align*}
Thus, choosing $A_1$ sufficiently large, applying the Cauchy-Schwarz
inequality to $\IP{\tfrac{\del g}{\del t}, \delb \gb + \del \bgb}$, and dropping some
negative terms we obtain
\begin{gather} \label{lteprop40}
 \begin{split}
 \square W_3 \leq&\ \brs{\tfrac{\del g}{\del t}}^2 -
\tfrac{A_1}{2} \brs{T}^2 - 2 A_2 \brs{\N^2 \Phi}^2.
\end{split}
\end{gather}
It follows from Proposition \ref{p:transgression} that $\tfrac{\del g}{\del t}$
can
be expressed as the $(1,1)$ projection of the $J$-Chern Hessian of $\Phi$.   Using the expression for the Chern connection it follows that there is a uniform
constant $C$ such that
\begin{align} \label{lteprop50}
 \brs{\tfrac{\del g}{\del t}}^2 \leq&\ C \left[ \brs{\N^2 \Phi}^2 + \brs{T}^2
\brs{\N \Phi}^2 \right].
\end{align}
Since $\brs{\N \Phi}^2$ is uniformly bounded by Proposition
\ref{f:gradientestimate},
we may use (\ref{lteprop50}) in (\ref{lteprop40}) to obtain, for $A_1$ and $A_2$
sufficiently large,
\begin{gather}
 \square W_3 \leq 0.
\end{gather}
An upper bound for $W_3$ follows by the maximum principle, giving then an upper bound for $\tfrac{\del f}{\del t}$, as required. The long time existence of GKRF is proved.

Now we prove the convergence of Mabuchi energy.  Arguing similarly to \cite[Proposition 6.11]{ASnondeg}, it follows that for any background metric $\til{g}$ one has
\begin{align*}
\brs{\brs{d \Phi}}_{L^2(\til{g})} \to 0.
\end{align*}
Thus, by the Poincar\'e inequality for $\til{g}$, one has \begin{align*}
\brs{\brs{ \Phi - \bar{\Phi}_{\til{g}}}}_{L^2(\til{g})} \to 0.
\end{align*}
where $\bar\Phi_{\til{g}}$ is the average of $\Phi$ with respect to $\til{g}$.
On the other hand, we estimate by Jensen's inequality that
\begin{align*}
e^{\gl} = \frac{\int e^{-\Psi_2}dV_g}{\int e^{-\Psi_1}dV_g} = \frac{\int e^{\Phi} e^{-\Psi_1}dV_g}{\int e^{-\Psi_1}dV_g} \geq \exp \left( \frac{\int \Phi e^{-\Psi_1}dV_g}{\int e^{-\Psi_1}dV_g} \right) = \exp \left( \frac{\int \left( \Phi - \bar{\Phi}_{\til{g}} \right) e^{-\Psi_1}dV_g}{\int e^{-\Psi_1}dV_g} + \bar{\Phi}_{\til{g}} \right).
\end{align*}
Rearranging this and using  $\int e^{-\Psi_1}dV_g$ is constant by Lemma \ref{l:ev-volume} yields
\begin{align*}
\bar{\Phi}_{\til{g}} \leq&\ \gl + C \int \brs{\Phi - \bar{\Phi}_{\til{g}}}e^{-\Psi_1}dV_g\\
\leq&\ \gl + C \int \brs{\Phi - \bar{\Phi}_{\til{g}}}e^{-\Psi_1}dV_g\\
\leq&\ \gl + O(t^{-1}),
\end{align*}
where the last inequality follows from that $\Psi_1$ is bounded below and $dV_g$ is uniformly bounded. A similar argument using Jensen's inequality yields
\begin{align*}
e^{-\gl} = \frac{\int e^{-\Psi_1}dV_g}{\int e^{-\Psi_2}dV_g} = \frac{\int e^{- \Phi} e^{-\Psi_2}dV_g}{\int e^{-\Psi_2}dV_g} \geq \exp \left( \frac{\int - \Phi e^{-\Psi_2}dV_g}{\int e^{-\Psi_2}dV_g} \right) = \exp \left( \frac{\int \left( - \Phi + \bar{\Phi}_{\til{g}} \right) e^{-\Psi_2}dV_g}{\int e^{-\Psi_2}dV_g} - \bar{\Phi}_{\til{g}} \right),
\end{align*}
thus similarly we have
\begin{align*}
\bar{\Phi}_{\til{g}} \geq&\ \gl - O(t^{-1}).
\end{align*}
Note that it follows easily from this and the invariance of $\int e^{-\Psi_2}dV_g$ from Lemma \ref{l:ev-volume} that
\begin{align*}
\mathcal M_1 =&\ \int_M \Phi e^{-\Psi_1}dV_g = \gl \int e^{-\Psi_1}dV_g + \int_M \left( \Phi - \gl \right) e^{-\Psi_1}dV_g = \gl \int e^{-\Psi_1}dV_g + O(t^{-1}).
\end{align*}
This finishes the case of $\mathcal M_1$, and $\mathcal M_2$ is similar.  The proposition is proved.
\end{proof}
\end{prop}

\section{Convergence in generalized K\"ahler class}

In this section we establish item $(2)$ of Theorem \ref{t:longcon} as Proposition \ref{p:conv} below, and prove the various corollaries stated in the introduction.  We briefly describe the idea of the proof.  The key extra assumption is that our initial data is connected to a K\"ahler Calabi-Yau metric through a path $(g_s, b_s, I, J_s)$ in the generalized K\"ahler class.  Let $\mathcal I\subset [0,1]$ be the set of $s\in [0,1]$ such that the GKRF with initial data $(g_s,b_s,I,J_s)$ converges smoothly to the Calabi-Yau structure $(g_0,b_0,I,J_0)$.  This set is obviously nonempty since $(g_0, b_0, I, J_0)$ is K\"ahler Calabi-Yau by assumption.  The fact that $\mathcal I$ is open follows from a stability result of \cite{HCF} and the fact that the GKRF depends smoothly on the initial data.  The crucial point is to show that $\mathcal I$ is also closed.  The key point is to show that a sequence of convergent flows has a uniform upper bound for the first time the flow enters a small ball around the Calabi-Yau structure.  This requires a certain backward regularity result for GKRF (Proposition \ref{Backwardregularity}), which crucially depends on uniform a priori decay of the gradient of the Ricci potential $\Phi$ (Proposition \ref{p:gradientdecay}).  

First we record a stability result for pluriclosed flow with initial data sufficiently close to a K\"ahler-Einstein metric.  This was originally shown in \cite{HCF} for the general class of Hermitian curvature flows, and also follows from recent stability results for generalized Ricci flow (\cite{VezzoniGRF,KHLeeGRF}).

\begin{prop}\label{stability}\cite[Theorem 1.2]{HCF} Let $(M^{2n}, g_{CY}, I)$ be a compact K\"ahler Calabi-Yau manifold.  There exists $\epsilon=\epsilon(\tilde g,k,\alpha)$ such that if $g$ is a Hermitian metric
on $(M, I)$ satisfying $|g-\tilde g|_{C^{k,\alpha}}\leq \epsilon$, for some integer $k>10$ and $\alpha\in(0,1)$, then the solution to pluriclosed flow with initial condition
$g$ exists for all time and converges to a K\"ahler Calabi-Yau metric.
\end{prop}

Now we proceed to prove the key backward regularity result. It will be helpful to introduce the so called harmonic radius  from Riemannian geometry, which is a useful tool to measure the regularity of Riemannian metric. 
%\section{Strong convergence of the GKRF}
\begin{defn}\label{d:harmonic_radius}
 Let $(M,g)$ be a Riemannian manifold and $x\in M$, 
 %we define the \emph{harmonic radius} $r_h(x,g)$ 
 %so that $r_h(x)=0$ if no neighborhood of $x$ is a Riemannian manifold. 
 %Otherwise, 
 we define the harmonic radius $r_h(x,g)$ at point $x$ to be the largest $r>0$ such that there exists a mapping $\Phi:B_r(0^n)\to X$ such that:
\begin{enumerate}
\item $\Phi(0)=x$ with $\Phi$ is a diffeomorphism onto its image.
\vskip1mm

\item $\Delta x^\ell = 0$, where $x^\ell$ are the coordinate functions and $\Delta$
 is the Laplace Beltrami operator.
\vskip1mm

\item If $g_{ij}=\Phi^*g$ is the pullback metric, then
\begin{align*}
||g_{ij}-\delta_{ij}||_{C^0(B_r(0^n))}+r\| \partial g_{ij}\|_{C^0(B_r(0^n))}
	 \leq 10^{-3}\, .
\end{align*}
\end{enumerate}
	We also define the harmonic radius $r_h(M,g)$ of $M$ as:  $$r_h(M,g):=\inf_{x\in M}r_h(x,g).$$
\end{defn}

Now we introduce the following useful concept, which measures how far a Hermitian metric is away from the standard flat K\"ahler metric on the Euclidean ball.

\begin{defn}
Let  $(M,g,I)$ be a Hermitian manifold and $H$ be the torsion tensor of the Bismut connection of $(g, I)$. The \emph{$k$-th \hr} of $M$ at $x$ is defined by
	\begin{align*}
		r^k(M,g,x):&= \min \left\{r_h(x,g)\,,\ \left( \sum_{i=0}^{k}|\nabla_g^i\Rm|^{\frac{1}{i+2}} (x) \right)^{-1},\ \left( \sum_{i=0}^{k} |\nabla_g^i H|^{ \frac{1}{i+1}}(x) \right)^{-1} \right\},\\
		r^k(M,g):&=\inf_{x\in M}r^k(M,g,x).
	\end{align*}
	We fix $k \geq 12$ and denote $r^{k}(M,g)$ by $r(M,g)$ for simplicity.
	\end{defn}
 
 Now we are ready to prove the estimate for the regularity scale.  It says that if a GKRF has sufficiently small gradient of Ricci potential and a lower bound on the regularity scale at some time, then it has a lower bound on the regularity scale for a definite amount of backwards time.
 
\begin{prop}\label{Backwardregularity}
%\textcolor{red}{revise} 
There exists a dimensional constant $\epsilon(n)$ such that 
if $(M^{2n}, g_t, b_t, I, J_t)$ is a solution to GKRF with holomorphically trivial canonical bundles on $[-4,0]$, satisfying
\begin{enumerate}
    \item $\sup_{M \times [-2,0]} \brs{\N \Phi} \leq \ge(n)$,
    \item $r(M, g_0) \geq 1$.
\end{enumerate}
	Then for every $t \in [-1, 0]$, we have $r(M,g_t) > \frac{1}{2}$.
\end{prop}

\begin{proof}  We argue by contradiction.  If the conclusion fails, we can find a sequence of solutions to GKRF $\{(M^i, g^i_t, b_t^i, I^i, J^i_t), -2 \leq t \leq 0\}$ with holomorphically trivial canonical bundles satisfying 
	\begin{align*}
		\sup_{M \times [-4,0]} \brs{\nabla\Phi_i} \le \epsilon_i\to 0, \qquad r(M, g^i_0) \geq 1,
	\end{align*}
	but which violates the conclusion.  We fix one element of the sequence and omit the index $i$.  We will pick a time interval from this flow by a point-picking argument.
By assumption there exists a time $t \in [-1,0]$ such that $r(g_t)\leq\frac{1}{2}$.  We can choose the latest time $t$ such that $r(g_t)=\frac{1}{2}$ and
	denote it by $t_1$.  Then we check the interval $[t_1-r(g_{t_1})^2, t_1]$.  If there is a time $t$ in this interval such that $r(g_t) \leq  \frac{1}{2} r(g_{t_1})$, then we pick the latest such time $t_2$ satisfying $r(g_{t_2})=\frac{1}{2^2}$.  We continue this process by induction to construct a sequence of times $t_k$ such that $r(g_{t_k}) = \frac{1}{2^k}$.  Observe that for arbitrary $k$ we have
		\begin{align*}
		& |t_{k+1}-t_k| \leq \frac{1}{2^k}, \qquad t_k \geq -\sum_{k=0}^{\infty}\frac{1}{2^k}\geq - 2.
	\end{align*}
 Since the regularity scale is bounded below on the smooth time interval $M\times [-2,0]$ and in each step the regularity scale drops by one half, this process will stop after finitely many steps.
So we may assume that the process stops at the $(j+1)$-step. Notice that $j\geq 0$ by our assumption.
	Therefore, for some time $t_{j+1} \in [t_j-r(g_{t_j})^2, t_j]$, we have
	\begin{align*}
		&r(g_{t_{j+1}})=\frac{1}{2}r(g_{t_j}), \qquad r(g_t) \geq \frac{1}{2} r(g_{t_{j+1}}), \quad \; \forall \; t \in \left[ t_{j+1}-r(g_{t_{j+1}})^2, t_{j+1} \right].
	\end{align*}
	Now we perform a rescaling as follows: let $\gl_j:=r(g_{t_j})$ and let $\tilde{g}(t)=\gl_j^{-2}g(t_j+\gl_j^2 t), \; s=\gl_j^{-2}(t_{j+1}-t_j)\in[-1,0)$.
	Then for the rescaled GKRF flow $\tilde{g}$, we have
	\begin{align}%\label{inequality}
		&r(\tilde{g}_0)=1, \label{E1}\\
		&r(\tilde{g}_s)=\frac{1}{2}, \label{E2} \\
		&r(\tilde{g}_t) \geq \frac{1}{2}, \quad \forall \; t \in [s,0], \label{E3}\\
		&r(\tilde{g}_t) \geq \frac{1}{4}, \quad \forall \; t \in \left[s-\frac{1}{4},s \right],\label{E4} \\
		&\sup_{M \times [-1,0]} \brs{\nabla\tilde\Phi}_{\til{g}} \le \epsilon \gl_j\le \epsilon,\label{E5}
	\end{align}
	where the last inequality is due to the scaling invariance  of $\Phi$.
	
	Now for each flow violating the conclusion fixed at the beginning of the proof, we perform this point picking process and also rescale the flow according to the above.  Denote the new GKRFs
	by $$\{(M^i, \tilde{g}^i_t, \tilde{b}^i_t, I^i, J^{i}_t), -1 \leq t \leq 0\}.$$  Then the  inequalities in (\ref{E1}-\ref{E5}) hold for each $\tilde{g}^i$ with some $s_i \in [-1,0)$ and $\epsilon_i \to 0$.
	Let $x^i$ be the point where $r(\til{g}^i_{s_i})$ achieves minimum value on $(M^i,\tilde g_i(s_i))$, so that in particular
	\begin{equation}\label{a}
	r(\tilde g^i_{s_i},x^i)=\frac{1}{2}.
	\end{equation}
	By choosing a further subsequence we may assume $s_i \to s$.  Then by (\ref{E4}), on the interval $[{s}-\frac{1}{4},0]$ we have a uniform lower bound of the regularity scale.  It follows from the Shi-type smoothing estimates for pluriclosed flow (cf.\,\cite{HCF, Streetsexpent}) that for any $l \geq 0$ one has
	\begin{align*}
		\sup_{M^i \times [s-\frac{1}{8}, 0]}  \sum_{j+k=l} \left( \brs{\nabla^j\frac{\partial^k}{\partial t^k} Rm} + \brs{\nabla^j\frac{\partial^k}{\partial t^k} H}\right) \leq C_l,
	\end{align*}
	where $C_l$ is independent of $i$.	
	Notice that since $J^i_t$ and $I^i$ are orthogonal with respect to $\tilde g^i_t$, we have that all covariant derivatives of $J^i_t$ and $I^i$ are bounded in terms of the torsion tensor and its derivatives.
	
	Combining all these facts, there is a subsequence converging smoothly in the pointed generalized Cheeger-Gromov sense on time interval $[{s}-\frac{1}{8}, 0]$:
	\begin{align*}
		\left\{ \left(M^i,x^i,\tilde{g}^i_t,\tilde{b}^i_t, I^i,J^i_t \right) \right\} \longrightarrow{}
		%\left\{
		\left(M^{\infty},x^{\infty},\tilde g^{\infty}_t,\tilde{b}_t^{\infty}, I^{\infty}, J^{\infty}_t \right).
		%\right\}.
	\end{align*}
	We note that the convergence above is up to generalized gauge transformation, allowing for diffeomorphism and $b$-field shift (cf. \cite[Ch. 5]{GRFbook}).
    Furthermore, by construction the limiting structure satisfies $\brs{\N \Phi^{\infty}}_{\til{g}^{\infty}} \equiv 0$, thus the limit flow is static.  Since the harmonic radius  is continuous under smooth convergence \cite{AC}, so is the regularity scale.
	By construction, $r(\tilde g^i_0)=1$, thus $r(\tilde{g}^{\infty}_0)=1$.
	Furthermore, since the limit flow $\tilde g^{\infty}_t$ is static, we have $r(\tilde g^{\infty}_t)=1$ for all $t\in[s,0]$.
	On the other hand, by (\ref{a}) and the smooth convergence, we have  $r(\tilde g^{\infty}_s,x^{\infty})<1$. This is a contradiction and hence the proposition is proved.
\end{proof}

\begin{prop} \label{p:conv}
Let $(g^s,b^s,I,J^s), s\in [0,1]$ be a smooth path in the generalized K\"ahler class $[(g^0, b^0, I, J^0)]$, where $(g^0,b^0,I,J^0)$ is K\"ahler Calabi-Yau.  Then for any $s\in [0,1]$, the GKRF with initial data $(g^s,b^s,I,J^s)$ exists on $[0,\infty])$ and converges to $(g^0,b^0,I,J^0)$.
\end{prop}

\begin{proof} 
It follows from Lemma \ref{l:spinorinvariance} that every GK structure $(g^s,b^s, I, J^s)$ has holomorphically trivial canonical bundles.  Thus the long time existence of solution to GKRF is established in Proposition \ref{lteprop}.  Let $\mathcal I\subset [0,1]$ be the set of $s\in I$ such that the GKRF with initial data $(g^s,b^s,I,J^s)$ converges up to exact $b$-field transformation to the K\"ahler Calabi-Yau structure $(g^0,b^0,I,J^0)$ along the GKRF.  We want to show that $\mathcal I=[0,1]$.  Trivially, we have $0\in \mathcal I$, so it suffices to show that $\mathcal I$ is both open and closed in $[0,1]$.

First we show that $\mathcal I$ is open.  Fix $s\in \mathcal I$ such that the GKRF with initial data $(g^s,b^s,I,J^s)$ converges smoothly to the K\"ahler Calabi-Yau structure $(g^0,b^0,I,J^0)$.  In particular, for any $\ge > 0$ there exists $t_0 > 0$ so that $(g^s_{t_0},b^s_{t_0},I,J^s_{t_0})$ lies in the $\epsilon$-ball in the $C^{k,\alpha}$ topology centered at $(g^0,b^0,I,J^0)$.  By the smooth dependence on the initial data,
there is a small constant $\eta>0$ such that when $s'\in(s-\eta,s+\eta)$, the GKRF with initial data $(g^{s'},b^{s'},I,J^{s'})$ will enter the $2\ge$-ball in the $C^{k,\ga}$ topology centered at $(g^0,b^0,I,J^0)$.  Then by the stability of Proposition \ref{stability}, the GKRF then converges smoothly to some Calabi-Yau structure $(\tilde g,\tilde{b},I,\tilde J)$.  By Theorem \ref{t:rigidity}, Calabi-Yau structures are unique in a fixed generalized K\"ahler class, thus $(\tilde g,\tilde{b},I,\tilde J) = (g^0,b^0,I,J^0)$.  This proves the openness.

Now we show that $\mathcal I$ is also closed.  Let $s_i$ be a sequence in $\mathcal I$ converging to $s_{\infty}$.  For each $i$, let $T_i$ be the first time that the GKRF starting at $(g^{s_i},b^{s_i},I,J^{s_i})$ hits the boundary of the $\epsilon$-ball in the $C^{k,\ga}$ topology centered at $(g^0,b^0,I,J^0)$.  The key point is to show that $\lim\sup {T_i} < \infty$.  Given this, by passing to a subsequence we may assume that $T_i\to T<\infty$.  We claim that $(g^{s_{\infty}}_T, b^{s_{\infty}}_T,I,J^{s_{\infty}}_T)$ is in the closed $2\epsilon$-ball in the $C^{k,\ga}$ topology centered at the Calabi-Yau structure $(g^0,b^0,I,J^0)$.  If not, then by continuity, there is a constant $\eta>0$ depending on $(g^{s_{\infty}}, b^{s_\infty}, I, J^{s_{\infty}})$ such that when $t\in (T-\eta,T+\eta)$, $(g^{s_0}_t, b^{s_0}_t, I, J^{s_0}_t)$ is not in the $2\epsilon$-ball. Then by the smooth dependence on the initial data, for $i$ sufficiently large, $(g^{s_i}_{T_i}, b^{s_i}_{T_i}, I, J^{s_i}_{T_i})$ is not in the $2\epsilon$-ball.  This is a contradiction, thus
$(g^{s_0}_T, b^{s_0}_T, I, J^{s_0}_T)$ is in the $2\epsilon$-ball, and we can apply Proposition \ref{stability} to finish the proof.

Finally we show that $\lim\sup T_i<\infty$ indeed holds.  We argue by contradiction and assume that $\lim T_i=\infty$ by passing to a subsequence.
Note that there is some uniform bound for $\Phi^s_0$ along the continuity path.  It then follows from Proposition \ref{p:gradientdecay} that there is a uniform constant $C$ such that for any $i$ and any time $t$ one has
$$\sup_{M \times \{t\}}|\nabla\Phi^{s_i}_t|<\frac{C}{t}.$$ Therefore if $T_i\to \infty$, we have that
\begin{align} \label{f:Phidecay}
\sup_{M \times \{T_i\}} \brs{\nabla\Phi^{s_i}_{T_i}}<\epsilon_i \to 0.
\end{align}
Given this, we may apply Proposition \ref{Backwardregularity} to obtain
	\begin{align*}
	    \inf_{M \times [T_i - \delta, T_i]} r_h(x) \geq c, \qquad \sup_{M \times [T_i-\delta, T_i]}  |\nabla^l Rm| \leq C_l,  \quad \forall \; l \geq 0,
	\end{align*}
for some uniform constants $c, C_l,\delta$ independent of $i$.  We note then that the sequence $(g^{s_i}_{T_i},b^{s_i}_{T_i},I,J^{s_i}_{T_i})$ converges in \emph{any} $C^{k,\ga}$ topology to a limit $(g^{\infty}, b^{\infty},I, J^{\infty})$.  On the one hand, it follows from (\ref{f:Phidecay}) that this limit is a Calabi-Yau structure, which must be $(g^0, b^0,I, J^0)$ by Theorem \ref{t:rigidity}.  On the other hand, by construction the structures $(g^{s_i}_{T_i}, b^{s_i}_{T_i}, I, J^{s_i}_{T_i})$ lie on the boundary of the $\ge$-ball in the  $C^{k,\alpha}$ topology centered at $(g^0,b^0,I,J^0)$.  This is a contradiction, thus $\lim\sup T_i<\infty$.
\end{proof}

\begin{thm}\label{t:longcon} Let $(M^{2n}, g, b, I, J)$ be a compact generalized K\"ahler manifold with holomorphically trivial canonical bundles.  
\begin{enumerate}
    \item Suppose $(M, I)$ is a K\"ahler manifold.  Then the solution to generalized K\"ahler-Ricci flow with initial data $(g, b, I, J)$ exists for all time, and the Mabuchi energies converge to their topologically determined extreme values.
    \item Suppose there exists a generalized Calabi-Yau geometry in $[(g, b, I, J)]$.  Then the solution to generalized K\"ahler-Ricci flow with initial data $(g, b, I, J)$ converges exponentially to this necessarily unique generalized Calabi-Yau geometry.
\end{enumerate}
\end{thm}

\begin{proof}%[Proof of Theorem \ref{t:longcon}]
The theorem follows from Propositions \ref{lteprop} and \ref{p:conv}.
\end{proof}

\begin{cor} \label{cor:contractibility} Let $(M^{2n}, g, I)$ be a K\"ahler Calabi-Yau manifold which is part of a generalized Calabi-Yau geometry $m = (g,b, I, J)$. Then
\begin{align*}
\mathcal{GK}^{I,\sigma}(m) \cong \star.
\end{align*} 
\end{cor}

\begin{proof} The proof requires a minor modification of item (2) of Theorem \ref{t:longcon} to claim that for any element of $\mathcal{GK}^{I,\sigma}(m)$, the GKRF with this initial data converges exponentially to a Calabi-Yau structure.  The only difference is that a smooth path in $\mathcal{GK}^{I,\sigma}(m)$ may deform the Aeppli class of $\gw_I$, as the family of two-forms $K_t$ defining the path need only be closed, not exact (cf.\,\cite{gibson2020deformation}).  As the background manifold is K\"ahler, Aeppli cohomology is canonically identified with deRham cohomology, and thus there is a unique Calabi-Yau metric in every Aeppli class, and the moduli space of such metrics is smooth.  Thus a smooth path in $\mathcal{GK}^{I,\sigma}(m)$ comes equipped with a smooth family of Calabi-Yau metrics, and the proof of Proposition \ref{p:conv} is easily adapted to show that the flow starting with initial data along this path will converge to the relevant Calabi-Yau metric exponentially.  Using the exponential convergence, it follows easily that the one-parameter family of maps $F_s, s \in [0,1]$ sending a given element of $\mathcal{GK}^{I,\sigma}(m)$ to the time $t = \frac{s}{1-s}$ flow of the GKRF contracts $\mathcal{GK}^{I,\sigma}(m)$ to the space of Calabi-Yau metrics on $(M, I)$, which is itself contractible.
\end{proof}

\begin{cor} \label{c:HKcontract} Let $(M^{4n}, g, I, J, K)$ be a hyperK\"ahler manifold.  Let
\begin{align*}
    \Ham^+(\gw_K) := \{ \phi \in \Ham(\gw_K)\ |\ \phi^*\gw_I(X,IX)>0\ \mbox{for nonzero }X\in TM\}.
\end{align*}
Then the connected component of the identity $\Ham^+_0(\gw_K)\subset \Ham^+(\gw_K)$ is contractible:
\begin{align*}
    \Ham^+_0(\gw_K) \cong \star.
\end{align*}
Furthermore, $\Ham^+_0(\gw_K)\cap \mathrm{Aut}(M,J)=\{\mathrm{id}\}$.
\end{cor}
\begin{proof}We fix a hyperK\"ahler manifold $(M,g,I,J,K)$ and consider the nondegenerate generalized K\"ahler structure $m_{\mathrm{HK}}:=(g,0,I,J)$ with Poisson tensor $\sigma=\gw_K^{-1}$.

According to~\cite[Prop.\,2.16]{ASnondeg} and~\cite[Ex.\,3.9]{gibson2020deformation} any element $\phi \in\Ham_0^+(\gw_K)$ defines a generalized K\"ahler structure $\phi\cdot m_{\mathrm{HK}}$ in the generalized K\"ahler class $[m_{\mathrm{HK}}]$
in the following way. We fix $I$, let $J_{\phi} = \phi^* J$, and define a metric $g$ by
\begin{align*}
    g_{\phi} = \left( \pi_{1,1}^I \phi^* \gw_I \right) I.
\end{align*}
The positivity condition for $\phi \in \Ham_0^+(\gw_K)$ is precisely equivalent to $g_{\phi}$ being positive definite, and the underlying $b$-field can be uniquely recovered from $I,J_\phi$ and $g_\phi$ as $b=-(I(F_+)_\phi)^{\mathrm{skew}}$, where $(F_+)_\phi = -2g(I+J_\phi)^{-1}$

It follows from the definition of the generalized K\"ahler class (Definition~\ref{d:GKclass}) that, conversely, any element in $[m_{\mathrm{HK}}]$ is obtained by the action of some $\phi\in \Ham_0^+(\gw_K)$. Indeed, any exact canonical family with $K_t=da_t$ must be of the form $K_t=dd^c_{J_t}u_t$, $u_t\in C^\infty(M,\R)$ by the $dd^c_{J_t}$-lemma. Any such canonical deformation has $J_t=\phi_t^*J$, where $\phi_t$ is the Hamiltonian $\gw_K$-symplectomorphism generated by $u_t$. Thus we have a surjective map $\Ham^+_0(\gw_K)\to [m_{\mathrm{HK}}]_J$,
\[ \phi\mapsto \phi\cdot J, \]
where $[m]_J$ denotes the $J$ component of $m$.
We claim that this map is a bijection:
\[
\Ham^+_0(\gw_K)\cdot m_{\mathrm{HK}}\simeq [m_{\mathrm{HK}}]_J.
\]
This in particular implies that $\Ham^+_0(\gw_K)\cap\mathrm{Aut}(M,J)=\{\mathrm{id}\}$.

Assume that $(\phi_0\cdot m_{\mathrm{HK}})_J=(\phi_1\cdot m_{\mathrm{HK}})_J$ and consider a path $\phi_s\in \Ham^+_0(\gw_K)$ between $\phi_0$ and $\phi_1$. We apply GKRF with the initial data $\phi_s\cdot m_{\mathrm{HK}}$, $s\in[0,1]$. The GKRF on a nondegenerate generalized K\"ahler manifold is given by an action of $\Ham(\gw_K)$, and smoothly depends on the initial data $\phi_s\cdot m_{\mathrm{HK}}$. Thus by Theorem~\ref{t:longcon} all the flows starting with $\phi_s\cdot m_{\mathrm{HK}}$ will converge to the unique Calabi-Yau geometry $m_{\mathrm{HK}}\in [m_{\mathrm{HK}}]$. Thus, the $J$-components of $\phi_s\cdot m_{\mathrm{HK}}$ will converge to a unique $J_{\rm HK}$ in the hyperK\"ahler family.

Let $\phi_{s,t}\in \Ham(\gw_K)$, $t\in[0,\infty)$ be the Hamiltonian isotopy provided by the GKRF connecting $\phi_s\cdot m_{\mathrm{HK}}$ to $m_\mathrm{HK}$. Since the GKRF flows through genuine generalized K\"ahler structure with positive definite $g$, we have $\phi_{s,t}\circ \phi_t\in \Ham_0^+(\gw_K)$. As the flow converges exponentially it follows that there exists a smooth limit $\phi_{s,\infty}$ satisfying
\[
(\phi_{s,\infty}\circ\phi_s)\cdot J_{\rm HK}=J_{\mathrm{HK}},
\]
implying that $\phi_{s,\infty}\circ\phi_s$ is a one-parameter family of $\gw_{K}$-Hamiltonian automorphisms of $(M,J_{\rm HK})$. A vector field generating such automorphism must be parallel with respect to the Levi-Civita connection of $g_{\rm HK}$, and cannot be Hamiltonian. Therefore $\phi_{s,\infty}\circ\phi_s=\mathrm{id}$. Since $\phi_{s,\infty}$ depends only on the initial data, we have $\phi_{0,\infty}=\phi_{1,\infty}$ and $\phi_0=\phi_1$ as claimed.

Now, once we have the identification
\[
\Ham^+_0(\gw_K)\simeq [m_{\mathrm{HK}}]_J,
\]
we can again invoke Theorem~\ref{t:longcon} and conclude that GKRF contracts $[m_{\mathrm{HK}}]_J$ onto $J_{\mathrm{HK}}$ and respectively $\Ham^+_0(\gw_K)$ onto $\mathrm{id}$.
\end{proof}

\begin{cor} \label{c:commutingcase} Let $(M^{2n}, g, b, I, J)$ be a compact generalized K\"ahler manifold satisfying
\begin{enumerate}
    \item $\gs = 0$,
    \item $(M, I)$ is K\"ahler and $c_1(M, I)=0$.
\end{enumerate}
Then the solution to generalized K\"ahler-Ricci flow with initial data $(g, b, I, J)$ exists for all time and converges to a K\"ahler Calabi-Yau metric.
\end{cor}

\begin{proof}
  To begin we recall some fundamental structure of GK manifolds with $\gs = 0$ (cf.\,\cite{ApostolovGualtieri}).  Here the operator $Q = IJ$ has eigenvalues $\pm 1$, and induces a holomorphic splitting of the tangent bundle into $TM = T_{+} \oplus T_-$ of $(M, I)$, with associated projection operators $\pi_{\pm} = \tfrac{1}{2} \left( 1 \pm Q \right)$.  Furthermore, since in this case $(M, I)$ is assumed K\"ahler we can assume without loss of generality that $H_0 = 0$.
  
  We first claim that there exists a GK structure $(g_{CY}, 0, I, J)$ such that both pairs $(g_{CY}, I)$ and $(g_{CY}, J)$ are K\"ahler, Calabi-Yau.  By assumption $(M, I)$ is K\"ahler and $c_1(M, I) = 0$ so there exist K\"ahler, Calabi-Yau metrics on $M$, and as in Proposition \ref{p:KCYspinors} we obtain a Beauville-Bogomolov decomposition on a finite cover, still denoted $M$, such that
  \begin{align*}
      M \cong X \times Y \times \mathbb T^{2k},
  \end{align*}
  where $X$ and $Y$ are compact simply-connected Calabi-Yau manifolds and $\mathbb{T}^{2k}$ is a torus.  This manifold also inherits a GK structure with $\gs = 0$, and furthermore
  \begin{align*}
      T_+ = TX \oplus V_+, \quad T_- = TY \oplus V_-, \quad T \mathbb{T}^{2k} = V_+ \oplus V_-.
  \end{align*}
  Note also that by a Bochner argument, for any K\"ahler Ricci flat metric on $M$, both projection operators $\pi_{\pm}$ are parallel, and thus so is $J$.  Thus it suffices to find a Calabi-Yau metric which is compatible with both $I$ and $J$.  Using the product decomposition above it suffices to find such a metric on the torus factor $\mathbb{T}^{2k}$.  Here since there exist generalized K\"ahler metrics on $M$ in the first place, we may pick any point $p$ and choose a metric on $T_p \mathbb{T}^{2k}$ which is compatible with both $I$ and $J$, and extend it using a global parallel frame to a flat metric on $\mathbb{T}^{2k}$ compatible with $I$.  However since $J$ is also parallel with respect to this metric, it follows that the metric is compatible with $J$ as well.  
  
  Now fix $(g, b, I, J)$ a generalized K\"ahler structure on $M$.  We first note that $(s g + (1 - s) g_{CY}, s b, I, J)$ is a one-parameter family of generalized K\"ahler structures.  We also claim that each of these has holomorphically trivial canonical bundles, after possibly passing to a finite cover.  Given this, the claim of global existence and convergence follows from the argument of Corollary \ref{cor:contractibility}.  After passing to a finite cover and using the splitting constructed above, we can find holomorphic volume forms $\Theta_{\pm}$ for the determinant bundles of $T_{\pm}$.  Any generalized K\"ahler metric of the above family splits orthogonally along $T_{\pm}$, so that $\gw_I^s = \gw_+^s + \gw_-^s$.  It follows that the spinors
\begin{align*}
    \psi_1 = \bar\Theta_+ \wedge e^{sb + \i \gw_-^s},\qquad \psi_2 = e^{sb + \i \gw_+^s} \wedge \bar{\Theta}_-
\end{align*}
algebraically determine the correct GK structure.  We claim that these are also closed.  Using the integrability conditions for generalized K\"ahler metrics of commuting type from \cite{ApostolovGualtieri}, one has
\begin{align*}
    d_{\pm} \gw_{\pm}^s = 0,
\end{align*}
where $d = d_+ + d_-$ is the splitting of $d$.  Furthermore, by construction
\begin{align*}
- s db = d^c_I \gw^s_I = \i \left( \delb_- \gw_+^s - \del_- \gw_+^s + \delb_+ \gw_-^s - \del_+ \gw_-^s \right),
\end{align*}
it follows from a straightforward computation that $d \psi_i = 0$.
\end{proof}

\begin{cor} \label{c:commutingddbar} Let $(M^{2n}, I)$ be a compact K\"ahler manifold with $c_1(M, I) = 0$ and suppose its tangent bundle  admits a holomorphic and involutive splitting $TM= T_+ \oplus T_-$.  Let  $J$ be the integrable complex structure which equals $I$ on $T_+$ and $-I$ on $T_-$.  Suppose $(g, b, I, J)$ and $(g',b',I,J)$ are generalized K\"ahler structures such that $[\gw_I]_A = [\gw'_I]_A$.  Then there exists $\phi \in C^{\infty}(M)$ such that
\begin{align*}
    \gw_I =&\ \gw_I' + \pi^{1,1}_I d J d \phi = \gw_I' + \i \left( \del_+ \delb_+ - \del_- \delb_- \right) \phi,
\end{align*}
where $\del_{\pm}, \delb_{\pm}$ are defined in terms of the $I$-invariant splitting  $TM=T_+ \oplus T_-$.
\end{cor}
\begin{proof} We note that since $(M, I)$ is K\"ahler there is a canonical identification between deRham and Aeppli cohomology.  With the given hyotheses, it follows from Corollary \ref{c:commutingcase} that the GKRF with initial data $(g, b, I, J)$ exists for all time and converges to a generalized K\"ahler Calabi-Yau geometry $(g_{CY}, b_{CY}, I, J)$, where $\gw_{CY} \in [\gw_I]_A$.  As discussed in \S \ref{s:GKRFbackground}, the GKRF flows along a canonical family driven by $\rho_I$, which by Proposition \ref{p:transgression} is of the form $\rho_I = -\tfrac{1}{2} d J d \Phi$.  Furthermore, it follows from the induced evolution equation for $g$ \cite[Proposition 3.6]{gibson2020deformation} that along the flow one has
\begin{align*}
    \dt \gw_I = - 2 \pi^{1,1}_I \rho_I = \pi^{1,1}_I d J d \Phi.
\end{align*}
Thus there exists a smooth function $f$ such that
\begin{align*}
    \gw_I = \gw_{CY} + \pi^{1,1}_I d J d f.
\end{align*}
Note that the argument applies equally to the data $(g',b',I,J)$, with limiting data $(g', b', I, J)$ satisfying $\gw_{CY}' = \gw_{CY}$ since $[\gw_I]_A = [\gw'_I]_A$.  Thus we obtain a potential $\phi$ as claimed.  Using the holomorphic and involutive splitting of the tangent bundle $TM = T_+ \oplus T_-$ %into directions where $I = \pm J$ (see \cite{ApostolovGualtieri}), 
we get a splitting $d = d_{+} + d_-$ which yields the second claimed equation for the potential $\phi$.
\end{proof}

\section{Appendix} \label{a:harmonictheory}
\subsection{Harmonic theory on manifolds with boundary}\label{ss:harmonic_bdry}

Here we recall some basic properties of Hodge theory with boundary, following~\cite{TaylorPDE1}.  Fix $(M, \del M, g)$ a Riemannian manifold with boundary, and let ${\mc H_l}(M,\Lambda^*(T^*M))$ be the space of differential forms on $M$ with finite Sobolev ${\mc H_l}$-norm, $l\geq 2$.  Furthermore define the space of forms satisfying the \emph{absolute} boundary conditions:
\begin{equation*}%\label{t:f:bdry}
	\begin{split}
	\mc H_l^A(M,\Lambda^*(T^*M))&=\{u\in {\mc H_l}(M,\Lambda^*(T^*M))\ |\ (\star_gu)|_{\del M}=0,\ (\star_gdu)|_{\del M}=0\}.
	\end{split}
\end{equation*}
Equivalently, $i_\nu u=i_{\nu}du=0$ where $\nu$ is the outer unit normal vector for $\del M$. This boundary condition is an appropriate modification of the Neumann boundary condition for scalar PDEs.  Next we say that a form $u\in \mc H_l^A(M,\Lambda^*(T^*M))$ is \emph{harmonic} if
\[
du=d^*u=0.
\]
It follows from standard elliptic estimates~\cite[Prop.\,9.7]{TaylorPDE1} that the space of harmonic forms is finite dimensional. For a form $u\in L^2(M,\Lambda^*(T^*M))$ we denote by $P_h^A(u)$ its projection onto the space of harmonic forms.

The key result from general elliptic theory~\cite[Ch.5 \S 9]{TaylorPDE1} is the existence of the Green's operator
\[
G^A\colon {\mc H_l}(M,\Lambda^*(T^*M))\to \mc H_{l+2}^A(M,\Lambda^*(T^*M))
\]
which inverts $\Delta$ on $(\mathrm{Im}\ P_h^A)^{\perp}$, i.e.,
\begin{equation}\label{t:f:hodge}
	u=\Delta G^A u+P_h^A(u)
\end{equation}
for any $u\in {\mc H_l}(M,\Lambda^*(T^*M))$.  A further important property is that with absolute boundary conditions, the Hodge decomposition preserves the space of exact forms.
\begin{lemma}\label{lm:app_harmonic}
	If $u=dv\in {\mc H_l}(M,\Lambda^*(T^*M))$ is exact, then 
	\[
	dG^Au=0,\quad P_h^A(u)=0
	\]and the Hodge decomposition reduces to
	\[
	u=dd^*G^Au.
	\]
\end{lemma}
\begin{proof}
	By the Hodge decomposition (\ref{t:f:hodge}) applied to $u=dv$, we have
	\[
	dv=dd^*G^Au+d^*dG^Au+P_h^A(u).
	\]
	Denote $w=dv-dd^*G^Au-P_h^A(u)=d^*dG^Au$. Clearly, $dw=0$. Now, integrating by parts we find
	\[
	||w||^2=\int_M \la {w,d^*dG^Au} dV_g = \int_{\del M}\la {w, i_{\nu}dG^Au} dS_g.
	\]
	Now, the range of $G^A$ is $\mc H_{l+2}^A$, so $i_\nu dG^A u$ vanishes.  Thus $0 = w = d^*dG^Au$.
	Again using $i_{\nu} d G^A u = 0$ we obtain
	\[
	||dG^Au||^2= \int_{\del M}\la{ i_{\nu}dG^Au,G^Au} dS_g=0.
	\]
	Finally, we prove that the harmonic part $P_h^A(u)$ vanishes:
	\[
	||P_h^A(u)||^2=\int_M \la {d(v-d^*G^Au),P_h^A(u)} dV_g = \int_{\del M}\la {(v-d^*G^Au),i_\nu P_h^A(u)} dV_g=0,
	\]
	where the last term vanishes since by definition $i_\nu P_h^A=0$.
\end{proof}

\subsection{Harmonic theory for the twisted Laplacian on manifolds with boundary}\label{ss:harmonic_gk_bdry}

The purpose of this section is to extend the setup of Section~\ref{ss:harmonic_bdry} to the \emph{twisted} differential operators  underlying a manifold $(M,H_0)$ equipped with a generalized metric $\GG\in \End(T\oplus T^*)$
\[
\GG=e^b\left(\begin{matrix}0 & g^{-1}\\g & 0\end{matrix} \right)e^{-b},
\]
where $b$ is a two form and $g$ is a Riemannian metric.  To this end, throughout Section ~\ref{ss:harmonic_bdry} we need to replace the usual de Rham differential $d$ with the twisted differential $d_{H_0}=d+H_0\wedge\cdot$. The relevant adjoint operator is defined as $d_{H_0}^*:=\star_{\GG} \circ d_{H_0}\circ \star_{\GG}^{-1}$, where the linear operator $\star_{\GG}\colon\Lambda^*(T^*M)\to\Lambda^*(T^*M)$ is determined using the generalized metric and Mukai pairing by
\[
(2\i)^n(\alpha,\star_{\GG}\beta) = (e^b\wedge\alpha)\wedge \star_g(e^b\wedge\beta)
\]
(factor $(2\i)^n$ appears due to a non-standard normalization of Mukai pairing in Definition~\ref{d:mukai}).  Finally, we define the twisted Laplace operator
\begin{align*}
\Delta_{H_0}=d_{H_0}d_{H_0}^*+d_{H_0}^*d_{H_0}.
\end{align*}

Below we sketch the argument establishing the elliptic theory for the twisted differential operators on a manifold with boundary. The argument is based on the general elliptic theory for the systems of PDEs developed in~\cite[\S 5.11]{TaylorPDE1}.  To simplify the notation,  we will apply a $b$-field transform to the metric $\GG$ (and the underlying spinors $\Lambda^*(T^*M)$).  This makes the operator $\star_{\GG}$ equal to the standard Riemannian Hodge star $\star_g$:
\[
(2\i)^n(u,\star_{\GG} v)=u\wedge \star_g v = \IP{u, v} dV_g
\]
Of course, this will change the background $H_0$ to $H_0-db$, nevertheless by abuse of notation, we will still denote the background 3-form by $H_0$. We observe that there is still the standard integration by parts identity
\[
\int_M \IP{u,d^*_{H_0}v} dV_g=\int_M \IP{d_{H_0}u,v} dV_g+\int_{\del M} \IP{u,i_\nu v} dV_g.
\]

We note that the operator 
\[
\Delta_{H_0}\colon C^\infty(M,\Lambda^*(T^*M))\to C^\infty(M,\Lambda^*(T^*M))
\] and the usual Hodge Laplacian
\[
\Delta_{g}\colon C^\infty(M,\Lambda^*(T^*M))\to C^\infty(M,\Lambda^*(T^*M))
\]
have the same principal symbol. Similarly, the boundary operators
\[
B_{H_0}\colon C^\infty(M,\Lambda^*(T^*M))\to C^\infty(\del M,\Lambda^*(T^*M))\oplus C^\infty(\del M,\Lambda^*(T^*M))
\]
\[
B_{H_0}(u)=(i_\nu u, i_\nu d_{H_0}u)
\]
and
\[
B_{g}\colon C^\infty(M,\Lambda^*(T^*M))\to C^\infty(\del M,\Lambda^*(T^*M))\oplus C^\infty(\del M,\Lambda^*(T^*M))
\]
\[
B(u)=(i_\nu u, i_\nu d u)
\]
also share the same principal symbol. By the general theory developed in~\cite[Ch.5 \S 11]{TaylorPDE1}, the operators $(\Delta_{H_0},B_{H_0})$  determine the \emph{regular elliptic boundary problem} (see \cite[Ch.5 Prop.\,11.9]{TaylorPDE1} for the precise definition). Now, it follows form \cite[Ch.5 Prop.\,11.16]{TaylorPDE1} that the map 
\[
T\colon \mc H_{l+2}(M,\Lambda^*(T^*M))\to \mc H_{l}(M,\Lambda^*(T^*M))\oplus \mc H_{l+3/2}(\del M,\Lambda^*(T^*M))\oplus \mc H_{l+1/2}(\del M,\Lambda^*(T^*M))
\]
defined by
\[
Tu=(\Delta_{H_0}u, B_{H_0}u)
\]
is Fredholm. This in turn implies the  elliptic regularity estimates necessary for the constructions of the Green's operator $G^A$ as well as the projection onto harmonic forms $P_h^A$ as in Section~\ref{ss:harmonic_bdry}. To conclude, we have the following result.

\begin{thm}\label{thm:app_green}
    Given a compact manifold $(M,H_0)$ with boundary and a generalized metric $\GG$ on $(M,H_0)$, let
    \[
    \mc H_l^A(M,\Lambda^*(T^*M)):=\{u\in\mc H_l(M,\Lambda^*(T^*M))\ |\ (\star_{\GG} u)\big|_{\del M}=(\star_{\GG} d_{H_0}u)\big|_{\del M}=0\}
    \]
    be the space of $\mc H_l$-regular differential form satisfying the \textit{absolute} boundary condition. Then there exists an operator
    \[
    G^A\colon \mc H_l(M,\Lambda^*(T^*M))\to \mc H_{l+2}^A(M,\Lambda^*(T^*M))
    \]
    such that for any $u\in \mc H_l(M,\Lambda^*(T^*M))$
    \[
    u=\Delta G^A u+P_h^A(u),
    \]
    where $v=P_h^A(u)\in C^\infty(M,\Lambda^*(T^*M))$ satisfies the same boundary condition and
    \[
    d_{H_0}v=d^*_{H_0}v=0.
    \]
\end{thm}
Lastly, following the proof of Lemma~\ref{lm:app_harmonic} it follows that the Hodge decomposition preserves the space of exact forms.
\begin{lemma}\label{lm:app_hamonic_twisted}
	If $u=d_{H_0}v\in {\mc H_l}(M,\Lambda^*(T^*M))$ is exact, then 
	\[
	d_{H_0}G^Au=0,\quad P_h^A(u)=0
	\]and the Hodge decomposition reduces to
	\[
	u=d_{H_0}d^*_{H_0}G^Au.
	\]
\end{lemma}

%\bibliographystyle{abbrv}
%\bibliography{Streets_Master_Bib}

\end{document}